\def\version{}%{\tiny Version June 1, 2016 (typeset: \today)}}

\documentclass[reqno,twoside,11pt]{amsart}
\usepackage{cite}
\usepackage{amsmath}
\usepackage{amsfonts}
\usepackage{amssymb}
\usepackage{epsfig}
\usepackage{verbatim}
\usepackage{color}

\usepackage{mathrsfs}
\usepackage{mathpazo}
\usepackage{mathabx}

\DeclareFontFamily{OT1}{eusb}{} \DeclareFontShape{OT1}{eusb}{m}{n} {<5> <6> <7> <8> <9> <10> <11> <12> <14.4> eusb10}{}
\DeclareMathAlphabet{\eusb}{OT1}{eusb}{m}{n}

\DeclareFontFamily{OT1}{eusm}{} \DeclareFontShape{OT1}{eusm}{m}{n} {<5> <6> <7> <8> <9> <10> <11> <12> <14.4> eusm10}{}
\DeclareMathAlphabet{\eusm}{OT1}{eusm}{m}{n}

\DeclareFontFamily{OT1}{eufm}{} \DeclareFontShape{OT1}{eufm}{m}{n} {<5> <6> <7> <8> <9> <10> <11> <12> <14.4> eufm10}{}
\DeclareMathAlphabet{\mathfrak}{OT1}{eufm}{m}{n}

\DeclareFontFamily{OT1}{fraktura}{}
\DeclareFontShape{OT1}{fraktura}{m}{n} {<5> <6> <7> <8> <9> <10> <11> <12> <13> <14.4> [1.1] eufm10}{}
\DeclareMathAlphabet{\fraktura}{OT1}{fraktura}{m}{n}

\DeclareFontFamily{OT1}{cmfi}{} \DeclareFontShape{OT1}{cmfi}{m}{n} {<5> <6> <7> <8> <9> <10> <11> <12> <13> <14.4> [0.9] cmfi10}{}
\DeclareMathAlphabet{\cmfi}{OT1}{cmfi}{b}{n}

\DeclareFontFamily{OT1}{cmss}{} \DeclareFontShape{OT1}{cmss}{m}{n} {<5> <6> <7> <8> <9> <10> <11> <12> <13> <14.4> cmss10}{}
\DeclareMathAlphabet{\cmss}{OT1}{cmss}{m}{n}
\newtheoremstyle{thm}{1.5ex}{1.5ex}{\itshape\rmfamily}{} {\bfseries\rmfamily}{}{2ex}{}

\newtheoremstyle{def}{1.5ex}{1.5ex}{\rmfamily\sl}{} {\bfseries\rmfamily}{}{2ex}{}

\newtheoremstyle{rem}{1.3ex}{1.3ex}{\rmfamily}{} {\bfseries\rmfamily}{}{2ex}{}

\newtheoremstyle{ass}{1.5ex}{1.5ex}{\rmfamily\sl}{} {\bfseries\rmfamily}{}{2ex}{}

\newenvironment{proofsect}[1] {\vskip0.1cm\noindent{\rmfamily\itshape#1.}}{\qed\vspace{0.15cm}}

\theoremstyle{thm}
\newtheorem{theorem}{Theorem}[section]
\newtheorem{lemma}[theorem]{Lemma}
\newtheorem{proposition}[theorem]{Proposition}

\newtheorem*{Main Theorem}{Main Theorem.}
\newtheorem{corollary}[theorem]{Corollary}

\newtheorem{definition}[theorem]{Definition}

\theoremstyle{rem}
\newtheorem{remark}[theorem]{{Remark}}

\numberwithin{equation}{section}

\renewcommand{\section}{\secdef\sct\sect}
\newcommand{\sct}[2][default]{\refstepcounter{section}
\addcontentsline{toc}{section}
{{\tocsection {}{\thesection}{\!\!\!\!#1\dotfill}}{}}
\vspace{0.7cm}
\centerline{ %\large
\scshape\arabic{section}.\ #1} \nopagebreak \vspace{0.2cm}}
\newcommand{\sect}[1]{
\vspace{0.4cm} \centerline{\large\scshape\rmfamily #1}
\vspace{0.2cm}}

\renewcommand{\subsection}{\secdef\subsct\sbsect}
\newcommand{\subsct}[2][default]{\refstepcounter{subsection}
\addcontentsline{toc}{subsection}
{{\tocsection{\!\!}{\hspace{1.2em}\thesubsection}{\!\!\!\!#1\dotfill}}{}}
\nopagebreak\vspace{0.45\baselineskip} {\flushleft\bf
\thesection.\arabic{subsection}~\bf #1.~}
\\*[3mm]\noindent
\nopagebreak}
\newcommand{\sbsect}[1]{
\vspace{0.1cm}\noindent
\textbf{#1.~}\vspace{0.1cm}}

\renewcommand{\subsubsection}{%
\secdef \subsubsect\sbsbsect}
\newcommand{\subsubsect}[2][default]{%
\refstepcounter{subsubsection} 
\addcontentsline{toc}{subsubsection}{{\tocsection{\!\!}
{\hspace{3.05em}\thesubsubsection}{\!\!\!\!#1\dotfill}}{}}
\nopagebreak
\vspace{0.15\baselineskip} \nopagebreak {\flushleft\rmfamily
\itshape\arabic{section}.\arabic{subsection}.\arabic{subsubsection}
\ \rmfamily #1\/.}\ }
\newcommand{\sbsbsect}[1]{\vspace{0.1cm}\noindent
\rmfamily \itshape
\arabic{section}.\arabic{subsection}.\arabic{subsubsection} \
\sffamily #1\/.\ }

\renewcommand{\caption}[1]{%
\vglue0.5cm
\refstepcounter{figure}
\begin{center}
\begin{minipage}[c]{0.8\textwidth}\small {\sc Fig.~\thefigure\ }#1\end{minipage}
\end{center}
}

\newcommand{\dist}{\operatorname{dist}}

\newcommand{\textd}{\text{\rm d}\mkern0.5mu}

\newcommand{\texte}{\text{\rm  e}\mkern0.7mu}

\newcommand{\Var}{\text{\rm Var}}

\newcommand{\WW}{\mathcal W}

\newcommand{\N}{\mathbb N}

\newcommand{\Q}{\mathbb Q}
\newcommand{\R}{\mathbb R}

\newcommand{\Z}{\mathbb Z}

\newcommand{\scrE}{\mathscr{E}}

\newcommand{\scrI}{\mathscr{I}}

\newcommand{\twoeqref}[2]{(\ref{#1}--\ref{#2})}

\newcommand{\cc}{{\text{\rm c}}}

\def\myffrac#1#2 in #3{\raise 2.6pt\hbox{$#3 #1$}\mkern-1.5mu\raise 0.8pt\hbox{$#3/$}\mkern-1.1mu\lower 1.5pt\hbox{$#3 #2$}}

\newcommand{\be}{\beta}
\newcommand{\ga}{\gamma}

\newcommand{\ep}{\epsilon}

\newcommand{\wt}{\widetilde}

\newcommand{\laweq}{\,\overset{\text{\rm law}}=\,}

\newcommand\independent{\protect\mathpalette{\protect\independenT}{\perp}}
\def\independenT#1#2{\mathrel{\rlap{$#1#2$}\mkern3mu{#1#2}}}

\newcommand{\frakp}{\mathfrak p}
\newcommand{\frakq}{\mathfrak q}

\begin{document}
%\vglue-0.4cm

\title[Oscillations in LRP distance \hfill \version\hfill]
{\large Arithmetic oscillations of the chemical distance\\in long-range percolation on~$\Z^d$}

\author[\hfill  \version \hfill Biskup and Krieger]
{Marek~Biskup \,and\, Andrew Krieger}
\thanks{\hglue-4.5mm\fontsize{9.6}{9.6}\selectfont\copyright\,\textrm{2022}\ \ \textrm{M.~Biskup, A.~Krieger.
Reproduction, by any means, of the entire
article for non-commercial purposes is permitted without charge.\vspace{2mm}}}
\maketitle

\vglue-5mm
\centerline{\textit{%$^1$
Department of Mathematics, UCLA, Los Angeles, California, USA}}

%\smallskip
%\centerline{\today}

\vskip4mm
\begin{quote}
\footnotesize \textbf{Abstract:}
We consider a long-range percolation graph on~$\mathbb Z^d$ where, in addition to the nearest-neighbor edges of~$\mathbb Z^d$, distinct~$x,y\in\mathbb Z^d$  are  connected by an edge independently with probability asymptotic to~$\beta|x-y|^{-s}$, for $s\in(d,2d)$, $\beta>0$ and~$|\cdot|$ a norm on~$\mathbb R^d$. We first show that, for all but a countably many~$\beta>0$, the graph-theoretical (a.k.a.~chemical) distance between typical vertices at $|\cdot|$-distance~$r$ is, with high probability as~$r\to\infty$, asymptotic to $\phi_\beta(r)(\log r)^\Delta$, where $\Delta^{-1}:=\log_2(2d/s)$ and~$\phi_\beta$ is a positive, bounded and continuous function subject to $\phi_\beta(r^\gamma)=\phi_\beta(r)$ for~$\gamma:=s/(2d)$. The proof parallels that in a continuum version of the model where a similar scaling was shown earlier by the first author and J.~Lin. This work also conjectured that~$\phi_\beta$ is constant which we show to be false by proving that $(\log\beta)^\Delta\phi_\beta$ tends, as~$\beta\to\infty$, to a non-constant limit which is independent of the specifics of the model. The proof reveals arithmetic rigidity of the shortest paths that maintain a hierarchical (dyadic) structure all the way to unit~scales.
\end{quote}

\section{Introduction and results}
%\nopagebreak\vglue-4mm
%\subsection{The model and main results}
%\nopagebreak
\noindent
The asymptotic behavior of the intrinsic, a.k.a.\ graph-theoretical or chemical, distance in random graphs has been a subject of intense research. A prime example is the first passage percolation of Hammersley and Welsh~\cite{Hamersley-Welsh} where edges of~$\Z^d$ are assigned random lengths and one is interested in the aggregate edge length $L(x,y)$ of the shortest path connecting~$x$ to~$y$. Under suitable mixing and moment assumptions, the Subadditive Ergodic Theorem (Kingman~\cite{Kingman1,Kingman2}) shows that $ x\mapsto   L(0,x)$ is, for~$x$  large, asymptotic to a   (deterministic)   norm on~$\R^d$ and, in particular, $L(0,nx)$ scales asymptotically linearly with~$n$. The conclusion extends to the chemical distance on the infinite cluster of supercritical bond percolation on~$\Z^d$ with $d\ge2$ (Antal and Pisztora~\cite{Antal-Pisztora}, Garet and Marchand~\cite{Garet-Marchand}). See the recent review by Auffinger, Damron and~Hanson~\cite{50years-FPP}.

Our focus in the present paper is on the chemical distance in long-range percolation graphs. More precisely, we will use long-range percolation as a means to add random shortcuts to the existing nearest-neighbor structure of~$\Z^d$. Our setting will be as follows: Given a collection of numbers $\{\frakq(x)\}_{x\in\Z^d}\subseteq[0,\infty]$ satisfying $\frakq(x)=\frakq(-x)$ for all~$x\in\Z^d$ and a parameter $\beta\in(0,\infty)$, set
\begin{equation}
\label{E:1.0}
\frakp_\beta(x,y):=1-\exp\bigl\{-\beta\,\frakq(x-y)\bigr\}
\end{equation}
 (where~$\texte^{-\infty}:=0$)  and consider the random graph with vertices~$\Z^d$ and an undirected edge between~$x$ and~$y$ present with probability~$\frakp_\beta(x,y)$, independently of other edges.  

The cases of prime concern for us are those when~$\frakq$ exhibits power-law decay which, in light of our use of~$\beta$ as an independent parameter, we take to mean
\begin{equation}
\label{E:1.1}
\frakq(x)\,\sim\,\frac1{|x|^s},\quad|x|\to\infty,
\end{equation}
for a norm~$|\cdot|$ on~$\R^d$ and a parameter~$s>0$. We  also  assume that~$\frakq(x)=+\infty$ whenever~$x$ is a neighbor of the origin to ensure that all the nearest-neighbor edges of~$\Z^d$ are present, and the graph is thus connected. The chemical distance $D(x,y)$ between~$x,y\in\Z^d$ is then defined as the minimal number of edges in any path connecting~$x$ to~$y$. 

Earlier studies have revealed five distinct parameter regimes of asymptotic scaling of the chemical distance with respect to the Euclidean metric: 
\settowidth{\leftmargini}{(1111)}
\begin{enumerate}
\item[(1)] $s<d$, where the percolation graph on all of~$\Z^d$ has finite intrinsic diameter (Benjamini, Kesten, Peres and Schramm~\cite{BKPS}),
\item[(2)] $s=d$, where the chemical distance grows logarithmically modulo log-log corrections (Coppersmith, Gamarnik and Sviridenko~\cite{CGS}),
\item[(3)] $d<s<2d$, where the chemical distance growth is polylogarithmic with exponents increasing from~$1$ to~$\infty$ as~$s$ varies from~$d$ to~$2d$ (Biskup~\cite{B1,B2}),
\item[(4)] $s=2d$, where the chemical distance has sublinear polynomial growth with a~$\beta$-dependent exponent (Benjamini and Berger~\cite{Benjamini-Berger}, Ding and Sly~\cite{Ding-Sly}),
\item[(5)] $s>2d$, where the asymptotically-linear scaling with the underlying metric on~$\Z^d$ valid  generically  for first passage percolation resumes (Berger~\cite{Berger-LRP}).
\end{enumerate}
Our focus here is on the intermediate regime $d<s<2d$. Here the early work~\cite{B1} by the first author showed
\begin{equation}
%\label{}
D(0,x)=(\log|x|)^{\Delta+o(1)}\,\text{ when }\,\frakq(x)=|x|^{-s+o(1)}\,\text{ as }\,|x|\to\infty
\end{equation}
where
\begin{equation}
\label{E:1.4u}
\Delta:=\frac1{\log_2(2d/s)}.
\end{equation}
The first author and J.~Lin~\cite{Biskup-Lin} then sharpened this to an asymptotic statement for a closely related continuum model with asymptotic decay \eqref{E:1.1}.  Our first item of business is to extend this conclusion to the model on~$\Z^d$.  
Write $B(x,r):=\{y\in\Z^d\colon|x-y|<r\}$
for a ball in~$|\cdot|$-norm and let $\#$ denote the counting measure on~$\Z^d$. We then have:

\begin{theorem}
\label{thm-1}
Let $d\ge1$ and~$s\in(d,2d)$ and assume~$\frakq$ obeys \eqref{E:1.1}. Write $\Delta$ for the quantity in \eqref{E:1.4u} and let~$\gamma:=\frac s{2d}$. For each~$\beta>0$ there exists a continuous function $\phi_\beta\colon(1,\infty)\to(0,\infty)$ subject to the log-log-periodicity condition
\begin{equation}
\label{E:1.3a}
\forall r>1\colon\quad\phi_\beta(r^\gamma)=\phi_\beta(r)
\end{equation}
and there is an at most countable set~$\Sigma\subseteq(0,\infty)$ such that, for all $\beta\in(0,\infty)\smallsetminus\Sigma$,
\begin{equation}
\label{E:1.4a}
\forall\epsilon>0\colon\quad
\frac1{r^d}\,\#\Biggl(\biggl\{x\in B(0,r)\colon \Bigl|\frac{D(0,x)}{\phi_\beta(r)(\log r)^\Delta}-1\Bigr|>\epsilon\biggr\}\Biggr)\,\,\underset{r\to\infty}{\overset{P}\longrightarrow}\,\,0.
\end{equation}
The map $\beta\mapsto\phi_\beta(r)$ is non-increasing and left-continuous. It is continuous at all~$\beta\not\in\Sigma$.
\end{theorem}

The function~$\phi_\beta$,  which plays the role of so called time constant in our setting,  depends on the whole set of connection parameters $\{\frakq(x)\}_{x\in\Z^d}$, but we keep that dependence implicit. Referring to \eqref{E:1.3a} as log-log-periodicity is justified by letting $\psi_\beta\colon\R\to(0,\infty)$ be defined by
\begin{equation}
\label{E:1.6a}
\psi_\beta(t):=\phi_\beta\bigl(\texte^{\gamma^{-t}}\bigr)
\end{equation}
and noting that \eqref{E:1.3a} then translates into (additive) $1$-periodicity of~$\psi_\beta$. The restriction to~$\beta\not\in\Sigma$ reflects on our inability to control the continuity of~$\beta\mapsto\phi_\beta(r)$.  Indeed, writing $L_\beta(r):=(\log r)^\Delta\phi_\beta(r)$ the proof actually gives
\begin{equation}
%\label{}
\frac1{\#B(0,r)}\, \#\biggl(\Bigl\{x\in B(0,r)\colon (1-\epsilon)L_{\beta^+}(r)\le D(0,x)\le (1+\epsilon)L_{\beta}(r)\Bigr\}\biggr)\,\underset{r\to\infty}{\overset P\longrightarrow}\,1,
\end{equation}
where $L_{\beta^+}(r):=\lim_{\beta'\downarrow\beta}L_{\beta'}(r)$. The latter then equals~$L_\beta(r)$ when~$\beta\not\in\Sigma$.

We note that in the continuum setting of~\cite{Biskup-Lin}, scaling arguments were used to show that~$\beta,r\mapsto\phi_\beta(r)$ is jointly continuous, which gave convergence for all~$\beta>0$. In addition, the convergence $D(0,rx)/L(r)\to1$ in~\cite{Biskup-Lin} was shown to hold in probability for every $x\ne0$. In this ``pointwise'' version, the mode of convergence cannot be improved to ``almost sure,'' at least in $d=1$. This is due to long edges offering effective shortcuts at and near the points where they land; see \cite[Observation 1.3]{Biskup-Lin} and~Fig~\ref{fig-1}.

%%%%%% FIGURE %%%%%%
\nopagebreak
%\vskip0.2cm
\begin{figure}[t]
\vglue-1mm
\centerline{\includegraphics[width=0.7\textwidth]{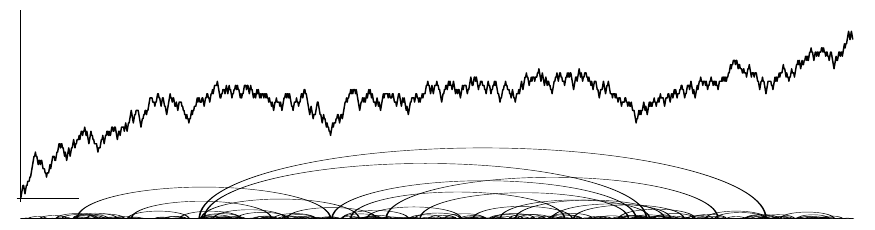}}
\vglue3mm
\centerline{\includegraphics[width=0.7\textwidth]{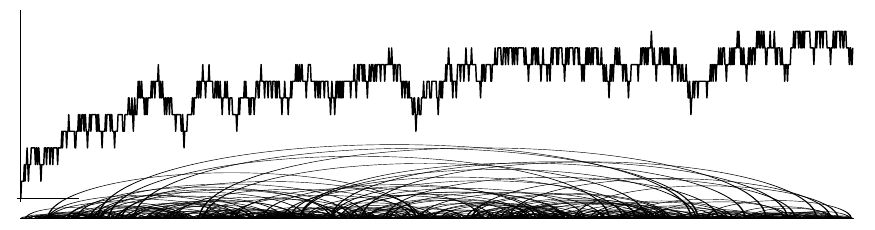}}
\begin{quote}
\small 
%\vglue-0.2cm
\caption{
\label{fig-1}
\small
Top figure: A plot of the chemical distance from the origin to points in $\{1,\dots,2000\}$ in a sample of long-range percolation on~$\Z$ with $\frakq(x):=|x|^{-s}$ (for~$|x|>1$) and parameters $s:=1.5$ and~$\beta:=1$. The arcs below depict the edges in the underlying graph. Bottom figure: A corresponding plot for a sample with $s:=1.5$ but $\beta:=5$. Note that the distance drops at the points where a long edge lands. The chemical distance plots are not to scale.}
\normalsize
\end{quote}
\end{figure}

Natural follow-up questions to \eqref{E:1.4a} are: What is~$\phi_\beta$? Can it be described more explicitly? What is its limit behavior as~$\beta\to\infty$ and~$\beta\downarrow0$? 
In~\cite{Biskup-Lin}, $\phi_\beta$ appeared to arise from the method of proof that was based  on subadditivity arguments along doubly-exponentially growing scales. In light of the  canonical scaling properties  of the continuum model, it seemed reasonable to conjecture that~$\phi_\beta$ is generally constant. However, as our next result shows, this is false. 

\begin{theorem}
\label{thm-2}
Let~$d\ge1$, $s\in(d,2d)$ and assume~$\frakq$ as above. Write~$\phi_\beta$ for the function from Theorem~\ref{thm-1}. Denote $m(\beta):=\sup\{k\in\Z\colon \gamma^{-k}\le\log\beta\}$ for $\gamma:=\frac s{2d}$ and set
\begin{equation}
\label{E:1.8a}
u(\beta):=\gamma^{m(\beta)}\log\beta\in[1,\gamma^{-1}).
\end{equation}
Define  $\psi_\beta$ from~$\phi_\beta$ via \eqref{E:1.6a}. Then for all $t\in[0,1]$,
\begin{equation}
\label{E:1.5}
(\log\beta)^\Delta\,\psi_{\beta}\biggl(t+\frac{\log(\frac{u(\beta)}{2d-s})}{\log(1/\gamma)}\biggr)\,\underset{\beta\to\infty}\longrightarrow\,\Bigl[\frac{s}{2d-s}(2\gamma)^{-t}-2\frac{s-d}{2d-s}2^{-t}\Bigr](2d-s)^\Delta
\end{equation}
with the limit uniform on~$[0,1]$.
\end{theorem}

As is readily checked,  the function on the right of \eqref{E:1.5} equals $(2d-s)^\Delta$ at $t=0$ and~$t=1$ (which is consistent with the $1$-periodicity of~$\psi_\beta$) yet, being the difference of two exponentials with distinct bases, it is not constant. We thus conclude:

\begin{corollary}
%\label{cor}
For each~$\frakq$ as above there is~$\beta_0\in(0,\infty)$ such that~$\phi_\beta$ is not constant for~$\beta>\beta_0$.
\end{corollary}

This refutes Conjecture~1.4 from~\cite{Biskup-Lin}  for the lattice version of the model.  The conclusion of Theorem~\ref{thm-2} also reveals the overall scaling of~$\phi_\beta$ with~$\beta$:

\begin{corollary}
%\label{cor}
For each $\frakq$ as above there are $c,C\in(0,\infty)$ and~$\beta_1>1$ such that
\begin{equation}
\label{E:1.10}
\forall \beta>\beta_1\,\forall r>1\colon\quad \frac{c}{(\log\beta)^\Delta}\le\phi_\beta(r)\le\frac{C}{(\log\beta)^\Delta}.
\end{equation}
\end{corollary}

The upshot of Theorem~\ref{thm-2} is that the asymptotic distances exhibit a \emph{universal} scaling limit as~$\beta\to\infty$ that depends only on the dimension~$d$ and the exponent~$s$ but not on the particulars of~$\frakq$. A plot of this limit, along with that for the asymptotic distance function $L_\beta(r):=\phi_\beta(r)(\log r)^\Delta$, is shown in Fig.~\ref{fig-2}.

The fact that~$\phi_\beta$ is not constant means that, for~$\beta$ large, $D(0,x)$ is sensitive to the \emph{arithmetic} nature of~$|x|$ --- namely, the fractional part of $\log_{1/\gamma}\log(|x|)$. The need for the $u(\beta)$-dependent term in \eqref{E:1.5} reveals that similar arithmetic oscillations occur also in $\beta$-dependence of the distance scaling function. As we explain below, these arise from the minimizing paths being rigid down to a lattice scale.

%%%%%% FIGURE %%%%%%
\nopagebreak
\vskip0.2cm
\begin{figure}[t]
\vglue-1mm
\centerline{\includegraphics[width=0.6\textwidth]{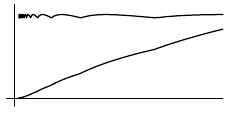}}
\begin{quote}
\small 
\vglue-0.3cm
\caption{
\label{fig-2}
\small
The graph of the $\beta\to\infty$ limit of $r\mapsto \phi_\beta(r)$ (top curve) and the corresponding limit for $L_\beta(r):=(\log r)^\Delta\phi_\beta(r)$ (bottom curve) for the choices $d=1$ and~$s=1.6$. (The adjustment due to~$u(\beta)$ is  absorbed into  the parametrization of the horizontal axis.) The bottom graph touches down on the horizontal axis at~$r=1$. The mild cusps at the points $\{\texte^{\gamma^n}\colon n\in\Z\}$ arise from the function on the right of \eqref{E:1.5} having unequal derivatives at~$t=0$ and~$t=1$.}
\normalsize
\end{quote}
\end{figure}

%\newpage
\section{Main ideas, connections and open questions}
\noindent
We proceed to review the main ideas of the proofs. The proof of Theorem~\ref{thm-1} runs very closely to that in \cite{Biskup-Lin} and so we focus on Theorem~\ref{thm-2}. We also highlight natural questions that we find worthy of further study and make connections to literature.

\subsection{Heuristics for distance oscillations}
\label{sec-2.1}\noindent
As shown in earlier work on this problem~\cite{B1,B2,Biskup-Lin}, the polylogarithmic scaling of the graph distance with the underlying metric on~$\Z^d$ in the parameter range $d<s<2d$ arises from a dyadic structure of the minimizing paths. Naturally, the larger the~$\beta$, the more long edges are there and the more rigid the dyadic structure should be expected to  be.  We will now present a semi-heuristic derivation of an asymptotic formula for the graph distance in the limit as~$\beta\to\infty$ that yields the conclusion \eqref{E:1.5}. This formula will be justified rigorously in later sections by way of asymptotically matching upper and lower bounds. 

Consider the long-range percolation on~$\Z^d$ with connection probabilities \twoeqref{E:1.0}{E:1.1} for some~$s\in(d,2d)$ and $\beta>0$. 
The aforementioned dyadic hierarchical structure of minimizing paths comes from the observation that, given two sites~$x$ and~$y$ with $N:=|x-y|\gg1$, the ball $B(x,N_1)$ is likely to contain an edge to the ball $B(y,N_2)$ provided that
\begin{equation}
\label{E:3.1a}
\beta\frac{(N_1N_2)^d}{N^s}\gg1
\end{equation}
while having even one such edge is  unlikely when the quantity on the left is~$\ll1$.  Writing (whenever such an edge exists)~$z$ for  the endpoint in $B(x,N_1)$ and~$z'$ for the endpoint in~$B(y,N_2)$, this yields  the  key \emph{subadditivity inequality}
\begin{equation}
\label{E:3.2a}
D(x,y)\le 1+D(x,z)+D(y,z')
\end{equation}
that drives all the recent work~\cite{B1,B2,Biskup-Lin}.
As it turns out, the inequality \eqref{E:3.2a} is actually saturated for at least one ``optimal'' choice of the edge $(z,z')$ where finding an optimal edge includes optimizing over the ``radii''~$N_1$ and~$N_2$ subject to \eqref{E:3.1a}.

Under the additional assumption that~$\beta\gg1$, these observations seem to point to the conclusion that $D(x,y)$ increases by one every time~$N:=|x-y|$ increases, roughly, through a specific power of~$\beta$. To see this note first that \twoeqref{E:1.0}{E:1.1} show that, for~$\beta$ large, vertices at $|\cdot|$-distance much smaller than~$\beta^{1/s}$ are very likely connected by a single edge while those at $|\cdot|$-distance much larger than~$\beta^{1/s}$ are quite unlikely to do so. Hence, with high probability, $D(x,y) =1$ when $0<N:=|x-y|\ll\beta^{1/s}$ and~$D(x,y)\ge2$ when~$N\gg\beta^{1/s}$. 
Proceeding inductively, if we assume that for each  $k=0,\dots,n$ there is~$\theta_k\ge0$ such that, with high probability once~$\beta$ is large, 
\begin{equation}
\label{E:2.3a}
D(x,y)\begin{cases}
\le k,\qquad&\text{for }N:=|x-y|\ll\beta^{\theta_k},
\\
\ge k+1,\qquad&\text{for }N\gg\beta^{\theta_k},
\end{cases}
\end{equation}
then  \twoeqref{E:3.1a}{E:3.2a} with $N_1:=\beta^{\theta_k}$ and~$N_2:=\beta^{\theta_{n-k}}$ yield~$D(x,y)\le n+1$ as long as, for at least one~$k\in\{0,\dots,n\}$,
\begin{equation}
\label{E:2.4w}
N\ll \beta^{1/s}(N_1N_2)^{d/s} = \beta^{\frac1s+\frac ds(\theta_k+\theta_{n-k})}.
\end{equation}
The fact that \eqref{E:3.2a} reduces to equality for an optimal choice of $(z,z')$ --- which  dictates the choice of~$k$ --- in turn gives $D(x,y)> n+1$ when~$N$ is much larger than the right-hand side for  every  $k\in\{0,\dots,n\}$. This reproduces the induction assumption \eqref{E:2.3a} for $k:=n+1$ provided~$\theta_{n+1}$ is set to the maximal exponent in \eqref{E:2.4w}. We are thus lead to:

\begin{definition}[Exponent sequence]
\label{def-1}
Let~$\{\theta_k\}_{k\ge0}$ be the sequence defined by the recursion
\begin{equation}
%\label{}
\theta_{n+1}:=\frac1s+\frac ds\max_{0\le k\le n}(\theta_k+\theta_{n-k})
\end{equation}
with initial value~$\theta_0:=0$.
\end{definition}

We will now make a couple of mathematical observations about this recursion and then solve it explicitly. First we note that the term $\theta_k+\theta_{n-k}$ is maximized by the ``most symmetric'' value of~$k$:

\begin{lemma}
\label{lemma-3.2}
Define the auxiliary sequence $\{\tilde\theta_n\}_{n\ge1}$ by $\tilde\theta_0:=0$ and, recursively,
\begin{equation}
\label{E:2.3}
\tilde\theta_{2n}:=\frac1s+\frac{d}s(\tilde\theta_n+\tilde\theta_{n-1})
\end{equation}
and
\begin{equation}
\label{E:2.4}
\tilde\theta_{2n+1}:=\frac1s+\frac{2d}s\tilde\theta_n.
\end{equation}
Then $n\mapsto\tilde\theta_n$ is non-negative and concave (on naturals) and, in fact,
\begin{equation}
\label{E:2.5}
\forall n\ge0\colon\quad\theta_n=\tilde\theta_n.
\end{equation}
\end{lemma}

\begin{proofsect}{Proof}
Non-negativity is immediate from the recursive definition. For concavity on naturals we note that, by \twoeqref{E:2.3}{E:2.4},
\begin{equation}
\label{E:2.6}
\forall n\ge1\colon\quad\tilde\theta_{2n+1}+\tilde\theta_{2n-1}-2\tilde\theta_{2n}=0
\end{equation}
while
\begin{equation}
\label{E:2.7}
\forall n\ge1\colon\quad\tilde\theta_{2n}+\tilde\theta_{2n-2}-2\tilde\theta_{2n-1}=\frac ds(\tilde\theta_n+\tilde\theta_{n-2}-2\tilde\theta_{n-1}).
\end{equation}
Since
\begin{equation}
%\label{}
\tilde\theta_2+\tilde\theta_0-2\tilde\theta_1 = \frac1s\Bigl(1+\frac ds\Bigr)-\frac2s = \frac{d-s}{s^2}<0
\end{equation}
we get $\theta_{n+1}+\theta_{n-1}-2\theta_n\le0$ for all~$n\ge1$ by induction. 

In order to prove \eqref{E:2.5} we note that the statement holds for~$n=0$ so, aiming for a proof by induction, let us assume $\theta_k=\tilde\theta_k$ for~$k=0,\dots,n$. Then
\begin{equation}
%\label{}
\theta_{n+1}=\frac1s+\frac ds\max_{0\le k\le n}(\theta_k+\theta_{n-k})
=\frac1s+\frac ds\max_{0\le k\le n}(\tilde\theta_k+\tilde\theta_{n-k})
\end{equation}
and the second maximum equals $2\tilde\theta_k$ if~$n=2k$ and $\tilde\theta_k+\tilde\theta_{k-1}$ if~$n=2k-1$ by concavity of~$n\mapsto\tilde\theta_n$. From \twoeqref{E:2.3}{E:2.4} we get~$\theta_{n+1}=\tilde\theta_{n+1}$ and so \eqref{E:2.5} follows by induction.
\end{proofsect}

The observations made in the previous proof are strong enough to identify the sequence $\{\theta_n\}_{n\ge1}$ explicitly:

\begin{lemma}
\label{lemma-2}
Recall the notation $\gamma:=\frac s{2d}$. The following holds for all~$n\ge0$:
\begin{equation}
\label{E:2.10}
\theta_{2^n-1}=\frac1s\frac{1-\gamma^n}{1-\gamma}\gamma^{-n+1}
\end{equation}
and, for all integers~$k$ satisfying $2^n-1\le k\le 2^{n+1}-1$,
\begin{equation}
\label{E:2.11}
\theta_{k} = \frac{2^{n+1}-1-k}{2^n}\,\theta_{2^n-1}+\frac{k-2^{n}+1}{2^n}\,\theta_{2^{n+1}-1}.
\end{equation}
In short, $k\mapsto\theta_k$ is piecewise linear with explicit values for~$k\in\{2^n-1\colon n\ge0\}$.
\end{lemma}

\begin{proofsect}{Proof}
We start with the explicit values. Note that $2^n-1$ is odd and $2(2^n-1)+1=2^{n+1}-1$. From \twoeqref{E:2.4}{E:2.5} we thus get
\begin{equation}
%\label{}
\theta_{2^{n+1}-1} = \frac1s+\gamma^{-1}\theta_{2^n-1} \,.
\end{equation}
Since \eqref{E:2.10} gives the correct value for~$n=0$, we get \eqref{E:2.10} for all $n\ge0$ by induction.

For \eqref{E:2.11} it suffices to prove that, for all~$n\ge1$,
\begin{equation}
\label{E:2.13}
\forall k\in\{2^n+1,\dots 2^{n+1}-1\}\colon
\quad
\theta_k-\theta_{k-1} = \theta_{k-1}-\theta_{k-2}
\end{equation}
because this shows that~$k\mapsto\theta_k$ is linear for $2^n-1\le k\le 2^{n+1}-1$ and so \eqref{E:2.11} follows from \eqref{E:2.10}. To show \eqref{E:2.13} note that, for~$k$ odd, the equality follows directly from \eqref{E:2.6} (and \eqref{E:2.5}). For~$k$ even, writing~$k=2\ell$ for some integer~$\ell$,  \eqref{E:2.7} (and \eqref{E:2.5}) tells us
\begin{equation}
%\label{}
\theta_k-\theta_{k-1} -(\theta_{k-1}-\theta_{k-2}) = \frac ds\Bigl(\theta_\ell-\theta_{\ell-1}-(\theta_{\ell-1}-\theta_{\ell-2})\Bigr).
\end{equation}
Since $2^n<k\le 2^{n+1}-1$ implies $2^{n-1}<\ell\le 2^n-1$, the right-hand side vanishes assuming that \eqref{E:2.13} holds for~$n-1$. Since \eqref{E:2.13} for~$n=1$ boils down to \eqref{E:2.6}, we get \eqref{E:2.13} for all~$n\ge1$ by induction.
\end{proofsect}

We will now present a calculation that determines the asymptotic in the main theorem based on the assumption that, for any~$x\in\Z^d\smallsetminus\{0\}$,
\begin{equation}
\label{E:3.17a}
D(0,x\beta^{\theta_n})\,\sim\, n
\end{equation}
where ``$\sim$'' means ``the ratio of the quantities tends to one in probability'' in the limits $n\to\infty$ followed by~$\beta\to\infty$. This assumption restates the conclusion \eqref{E:2.3a} of our heuristic reasoning while allowing for sublinear corrections.

First note that, for $\lambda\in[0,1]$ such that $2^n\lambda$ is an integer, \eqref{E:3.17a} yields
\begin{equation}
\label{E:3.18a}
D(0,x\beta^{\theta_{\lambda 2^n+(1-\lambda)2^{n+1}}}) \sim \lambda 2^n+(1-\lambda)2^{n+1}=\bigl[\lambda+(1-\lambda)2\bigr]2^n.
\end{equation}
Lemma~\ref{lemma-2} along with $\gamma^n\to0$ and $(2\gamma)^{-n}\to0$  in turn  give
\begin{equation}
\label{E:2.15}
\theta_{\lambda 2^n+(1-\lambda)2^{n+1}} \sim \lambda\theta_{2^n}+(1-\lambda)\theta_{2^{n+1}}
\sim\bigl[\lambda+(1-\lambda)\gamma^{-1}\bigr]\frac{1}{2d-s}\gamma^{-n}.
\end{equation} 
Theorem~\ref{thm-1} shows $D(0,x r^{\gamma^{-n}})/2^n\to L_\beta(r):=\phi_\beta(r)(\log r)^\Delta$ in measure as $n\to\infty$ provided~$\beta$ is not one of the exceptional values (which we can ignore thanks to the monotonicity of $\beta\mapsto L_\beta(r)$). From \twoeqref{E:3.18a}{E:2.15} and the continuity of~$r\mapsto\phi_\beta(r)$ we then get
\begin{equation}
\label{E:2.17}
L_\beta\bigl(\beta^{[\lambda+(1-\lambda)\gamma^{-1}]\frac1{2d-s}}\bigr)\sim \lambda+(1-\lambda)2
\end{equation}
and so
\begin{equation}
\label{E:2.23}
\phi_\beta\bigl(\beta^{[\lambda+(1-\lambda)\gamma^{-1}]\frac1{2d-s}}\bigr) \sim \frac{\lambda+(1-\lambda)2}{\,[\lambda+(1-\lambda)\gamma^{-1}]^\Delta}(2d-s)^\Delta\,\frac1{(\log\beta)^\Delta}.
\end{equation}
Writing $u(\beta)$ for the unique number in $[1,\gamma^{-1})$ such that $\beta=\texte^{u(\beta)\gamma^{-m(\beta)}}$ for a suitable integer~$m(\beta)$ (see \eqref{E:1.8a}),
the log-log-periodicity of~$\phi_\beta$ in \eqref{E:1.3a} then tells us
\begin{equation}
\label{E:2.24a}
\phi_\beta\bigl(\beta^{[\lambda+(1-\lambda)\gamma^{-1}]\frac1{2d-s}}\bigr) = \phi_\beta\bigl(\texte^{[\lambda+(1-\lambda)\gamma^{-1}]\frac {u(\beta)}{2d-s}}\bigr).
\end{equation}
Finally, let $t\in[0,1]$ be  the unique number  such that
\begin{equation}
%\label{}
\lambda+(1-\lambda)\gamma^{-1} = \gamma^{-t}.
\end{equation}
This is solved for~$\lambda$ by
\begin{equation}
\label{E:2.26}
\lambda = \frac{2d}{2d-s}-\frac s{2d-s}\gamma^{-t}.
\end{equation}
Using $\gamma^{-\Delta}=2$ we get $[\lambda+(1-\lambda)\gamma^{-1}]^\Delta=\gamma^{-t\Delta}=2^t$ while \eqref{E:2.26} shows
\begin{equation}
%\label{}
\lambda+(1-\lambda)2 = 2-\lambda = \frac{s}{2d-s}\gamma^{-t}-2\frac{s-d}{2d-s}.
\end{equation}
Inserting this into \twoeqref{E:2.23}{E:2.24a}, we get that, for all $t\in[0,1]$,
\begin{equation}
\label{E:2.24}
\phi_\beta\bigl(\texte^{\gamma^{-t}\frac {u(\beta)}{2d-s}}\bigr)\sim 2^{-t}\Bigl[\frac{s}{2d-s}\gamma^{-t}-2\frac{s-d}{2d-s}\Bigr]\,(2d-s)^\Delta\,\frac1{(\log\beta)^\Delta}.
\end{equation}
Writing the left-hand side using~$\psi_\beta$, we obtain \eqref{E:1.5}.

\subsection{Remarks and connections}
We proceed with some remarks on directions of possible future study as well as pointers to  relevant literature.

\smallskip
\noindent
(1) \textsl{Non-constancy for all~$\beta$ and extension to percolation setting}: Our proof of non-const\-ancy of~$\phi_\beta$ applies only to large~$\beta$ but we expect~$\phi_\beta$ to be non-constant for all $\beta>0$. An interesting starting point could be the $\beta\downarrow0$ asymptotic of~$\phi_\beta$, for which the continuum limit analyzed in \cite{Biskup-Lin} should be quite relevant.

Another extension concerns replacing the requirement that $\frakp_\beta(\cdot,\cdot)=1$ for  nearest  neighbors by the requirement that the graph contain an infinite connected component. We expect the asymptotic \eqref{E:1.4a} to take place here as well but several parts of the proof require new arguments.

\smallskip
\noindent
(2) \textsl{Subleading terms and ``shape theorem''}: The fact that~$\phi_\beta$ is non-constant complicates the ultimate goal of the whole sequence of works~\cite{B1,B2,Biskup-Lin}, which is to prove a ``shape theorem'' for balls of very large radii in the intrinsic (i.e., graph-theoretical or chemical) distance. Shape theorems lie at the core of the study of the First Passage Percolation; cf.\  Auffinger, Damron and~Hanson~\cite{50years-FPP}. The difficulty of the present situation is that the limit  of $D(0,nx)/L_\beta(n)$ as $n\to\infty$ is independent of~$x$ (as long as $x\ne0$).  The ``shape'' of the intrinsic ball, if there is  in fact  one at all, is  thus  determined by terms beyond the leading-order scale. 

In \cite[Conjecture~1.5]{Biskup-Lin}, a proposal for the relevant $x$-dependent second-order term was made but that only under the assumption that~$\phi_\beta$ is constant. We expect that the oscillations of~$\phi_\beta$ will contribute another such term, albeit perhaps of a smaller order. 

\smallskip
\noindent
(3) \textsl{Diameter scaling}: Another interesting question is the asymptotic scaling of the intrinsic diameter of large sets, e.g., lattice boxes or Euclidean balls. Our control of the point-to-point distance is too weak to rule out exceptional points --- which do exist, e.g., at endpoints (or points nearby) certain long edges. The main result of~\cite{B2} shows that the polylogarithmic exponent remains in effect for the diameter as well; the question is whether exceptional events may lead to sub-logarithmic corrections. 

\smallskip
\noindent
(4) \textsl{Other aspects of $d<s<2d$ regime}: Long range percolation in the regime of exponents considered in the present paper is attractive for other reasons than just those explored here. One of these is connectivity as a percolation model (again, dropping the requirement that~$\frakp(x)=1$ for~$x$ being a nearest neighbor). Here the $d<s<2d$ regime of \eqref{E:1.1} identifies a robust family of percolation models for which we have a proof of no percolation at criticality; see Berger~\cite{Berger-RW} and the recent work of Hutchcroft~\cite{Hutchcroft}. 

A somewhat opposite situation occurs for random walks on long-range percolation graphs (even with nearest-neighbor edges present). There an invariance principle (i.e., scaling to non-degenerate Brownian motion) is expected to hold for all exponents $s>d+2$ yet the method of proof breaks down when $s\le 2d$ due to the fact that the so called corrector fails to be sublinear everywhere (Biskup, Chen, Kumagai and Wang~\cite{BCKW}). The geometric aspects of long-range percolation such as those studied here will likely play an important role in extending the proof of  the  invariance principle to all exponents $s>d+2$. We refer to, e.g., Berger~\cite{Berger-RW}, Benjamini, Berger and Yadin~\cite{BBY}, Crawford and Sly~\cite{CS1,CS2}, Misumi~\cite{Misumi}, Kumagai and Misumi~\cite{KM}, Can, Croydon and Kumagai~\cite{CCK} for studies of random walk in long-range percolation setting and further connections.

\smallskip
\noindent
(5) \textsl{Inhomogenous percolation models}: In~\cite{DvdHH}, Deijfen, van der Hofstad and Hooghiemstra introduced an inhomogeneous version of the long-range percolation model where an edge between~$x$ and~$y$ is added with probability
\begin{equation}
%\label{}
1-\exp\Bigl\{-\beta\frac{w_x w_y}{|x-y|^s}\Bigr\},
\end{equation}
for a given collection~$\{w_x\}_{x\in\Z^d}$ of non-negative i.i.d.\ random weights. The main novelty here is that, by tuning the law of the $w$'s ---  specifically,  choosing it heavy tailed with a suitable exponent --- one can make the degree distribution of the graph ``scale free,'' which is an aspect relevant for real-life networks.

The appearance of another tunable parameter --- namely, the distribution of the~$w$'s or the relevant exponent therein  --- makes the ``phase diagram'' of the model more intricate (see Deprez, Hazra and W\"ut\-trich~\cite{DHW}, Heydenreich, Hulshof and Jorritsma~\cite{HHJ}, Hao and Heidenreich~\cite{HH}) although the five basic regimes of behavior outlined early in this section persist. It is of interest to explore whether the sharp leading-order asymptotic of the distance established here and \cite{Biskup-Lin} extend to the inhomogeneous case as well.

Another modification of our model comes in the work of Chatterjee and Dey~\cite{CD} in which an edge between~$x$ and~$y$ is assigned an exponentially distributed passage time with mean $|x-y|^{+s}$; one is then interested in the minimal passage time in paths connecting two vertices. Also here the regime $d<s<2d$ is of significance, being marked by stretched-exponential growth of the passage time with the $\ell^1$-distance. A novelty here is the appearance of an additional regime $2d<s<2d+1$, in which the passage time grows polynomially (as opposed to linear growth that resumes for $s>2d+1$).

\section{Proof of Theorem~\ref{thm-1}}
\noindent
The proof of Theorem~\ref{thm-1} follows closely that of its continuum predecessor \cite[Theorem~1.2]{Biskup-Lin}. Many steps of the proof can in fact be taken over nearly \emph{verbatim}; the main novelty is the need for a coupling between the lattice and continuum edge processes and an argument by-passing  potential  discontinuity points of $\beta\mapsto\phi_\beta(r)$.

\subsection{Subadditivity inequality}
 Fix $d\ge1$, $s\in(d,2d)$, $\beta>0$ and~$\frakq$ satisfying \eqref{E:1.1} throughout the rest of this section. 
Given a sample of the percolation graph  on~$\Z^d$,  let~$\scrE$ denote for the set of all occupied edges, of both orientations and including the nearest-neighbor ones, contained therein.
 Echoing  definition~(2.1) of~\cite{Biskup-Lin}  we  introduce $\wt D\colon \Z^d \times \Z^d \to \Z$ via
\begin{equation} 
\label{E:dist-1}
    \wt D(x,y) := \inf \left\{ n \ge 0 \colon \,
        \begin{aligned}
            &\{(x_{k-1},x_k) \colon k=1,\dotsc,n\} \subseteq \scrE , \:
            x_0 = x, \:
            \\
            &x_n = y, \,\forall k=1,\dots, n\colon\, |x_k-x|<2|x-y|_1 \\
        \end{aligned} \right\} .
\end{equation}
We will refer to $\wt D(x,y)$ as the \emph{restricted distance} from~$x$ to~$y$ as it is non-negative, strictly positive for~$x \ne y$ and arises by optimizing lengths of paths, although~$\wt D$ is not a distance in proper sense as it is not symmetric in general. What matters in the sequel is
\begin{equation}
\label{E:3.2}
\forall x,y\in\Z^d\colon\quad D(x,y)\le\wt D(x,y)\le|x-y|_1
\end{equation}
and the fact that the law of~$\wt D$ is translation invariant with
\begin{equation} \label{E:restr-indep}
    \forall x,y,x',y' \in \Z^d \colon \:
    |x-x'|_1 > 2|x-y|_1 + 2|x' - y'|_1
    \,\,\Rightarrow\,\,
    \wt D(x,y) \independent \wt D(x', y') .
\end{equation}
Here and henceforth $|\cdot|_1$ denotes the $\ell^1$-norm on~$\R^d$.

Let $\lfloor x\rfloor$, for $x\in\R^d$, denote the unique~$z\in\Z^d$ such that $x-z\in [0,1)^d$.
The independence property \eqref{E:restr-indep} enabled by the consideration of the restricted distance permits us to prove the following analogue of~\cite[Proposition~2.7]{Biskup-Lin} that drives the bulk of the subsequent derivations in this paper. 

\begin{proposition}[Subadditivity inequality] 
\label{P:subadd}
Fix~$\eta \in (0,1)$ and $\overline\gamma\in(\gamma,1)$. Let~$Z,Z'$ be i.i.d.~$\R^d$-valued random variables
with common law given by
\begin{equation} 
\label{E:z-law}
    P (Z \in B) = \sqrt{\eta \be} \int_B \texte^{-\eta \be c_0 |z|^{2d}} \textd z ,
\end{equation}
where
\begin{equation}
\label{E:3.5}
    c_0 := \int 1_{\{|z|^{2d} + |\tilde z|^{2d} \le 1\}}\,\textd z\textd\tilde z .
\end{equation}
Let~$\wt D'$ be an independent copy of~$\wt D$
with~$\wt D$ and~$\wt D'$ assumed independent of~$Z$ and~$Z'$.
For each $\gamma_1,\gamma_2 \in (0,\overline\gamma)$ with~$\gamma_1+\gamma_2=2\gamma=s/d$,
there are~$c_1,c_2 \in (0,\infty)$ and, for each~$x \in \Z^d$,
there is an event~$A(x) \in \sigma(Z,Z')$ such that
\begin{equation} \label{E:subadd}
    \wt D(0,x)
    \,\overset{\text{\rm law}}\le\, \wt D \bigl( 0, \lfloor |x|^{\gamma_1} Z \rfloor \bigr)
    + \wt D' \bigl( 0, \lfloor |x|^{\gamma_2} Z' \rfloor \bigr)
    + 1 + |x|_1 1_{A(x)}
\end{equation} 
and
\begin{equation} 
\label{E:A-prob}
    P \bigl( A(x) \bigr) \le c_1 \texte^{-c_2 |x|^\vartheta}
\end{equation}
hold with~$\vartheta := 2d[\overline\gamma - \max\{\gamma_1,\gamma_2\}]$.
\end{proposition}

\begin{proofsect}{Proof}
Fix $\eta\in(0,1)$, $\overline\gamma\in(\gamma,1)$ and $\gamma_1,\gamma_2\in(0,\overline\gamma)$ with $\gamma_1+\gamma_2=2\gamma$. Let~$x\in\Z^d$.  There is nothing to prove when $x=0$ so let us assume $x\ne0$. 
Following the overall strategy of the proof in~\cite{Biskup-Lin}, consider Borel measures~$\mu$ and~$\mu'$ on~$\R^d\times\R^d$ defined by
\begin{equation}
\label{E:3.8}
    \mu(\textd \tilde x \textd \tilde y) := \eta \be 1_{\{|\tilde x| < |\tilde y|\}}
        1_{\{|\tilde x| \vee |\tilde y-x| \le |x|^{\overline\gamma}\}}
        \frac{\textd \tilde x \textd \tilde y}{|x|^s}
\end{equation}
and
\begin{equation}
\label{E:3.9}
    \mu'(\textd \tilde x \textd \tilde y) := \eta \beta \frac{\textd \tilde x \textd\tilde y}{|x|^s} - \mu(\textd \tilde x \textd \tilde y).  
    \end{equation}
Next observe that, for~$|x|$ larger than an $\eta$-dependent constant, for any~$(\tilde x,\tilde y) \in \R^d\times\R^d$,
the inequalities~$|\tilde x| \le |\tilde y|$, $|\tilde x| \le |x|^{\overline\gamma}$,
and~$|\tilde y-x| \le |x|^{\overline\gamma}$ imply
\begin{equation}
\label{E:3.12}
        \bigl( \tfrac{1+\eta}2 \bigr)^{1/s} \, |x| <|\tilde x - \tilde y|<
        \bigl( \tfrac{1+\eta}2 \bigr)^{-1/s} \, |x| 
\end{equation}
and so, by the inequality on the right,  for all $\tilde x,\tilde y\in\Z^d$, 
\begin{equation} \label{E:dis-ctx-couple}
    \mu \Bigl( \bigl( \lfloor\tilde x\rfloor + [0,1)^d \bigr)
        \times \bigl( \lfloor\tilde y\rfloor + [0,1)^d \bigr) \Bigr)
    \le \frac{\eta\be}{|x|^s}
    \le \frac\eta{\bigl( \tfrac{1+\eta}2 \bigr)} \, \frac{\be}{|\tilde x-\tilde y|^s}. 
    %\le \frakp_\be(\lfloor\tilde x\rfloor-\lfloor\tilde y\rfloor)
\end{equation}
By \twoeqref{E:1.0}{E:1.1}, the left-inequality in \eqref{E:3.12} and $\eta (\tfrac{1+\eta}2)^{-1} < 1$, this is less than~$\frakp_\be(\lfloor\tilde x\rfloor-\lfloor\tilde y\rfloor)$ as soon as~$|x|$  exceeds a constant that depends on~$\eta$, $\beta$ and~$\frakq$. Under these circumstances we can  couple a Poisson point process~$\scrI$ with intensity measure~$\mu'$ to the discrete edge set~$\scrE$ so that
\begin{equation}
\label{E:3.11}
\forall (\tilde x,\tilde y) \in \scrI\colon\,\, (\lfloor \tilde x \rfloor, \lfloor \tilde y \rfloor) \in \scrE
\end{equation}
holds pointwise and, by \eqref{E:restr-indep} and  the  restriction built into the definition of~$\wt D$, the families
\begin{equation}
\label{E:3.13}
\bigl\{\wt D(0,\lfloor\tilde x\rfloor)\colon |\tilde x|\le |x|^{\overline\gamma}\bigr\},\,\bigl\{\wt D(x,\lfloor\tilde y\rfloor)\colon |\tilde y-x| \le |x|^{\overline\gamma}\bigr\},\,\scrI
\end{equation}
are independent.

Let~$\scrI'$ be a Poisson point process with intensity measure~$\mu'$ independent of~$\scrI$ and~$\scrE$. Then~$\scrI \cup \scrI'$ is a homogeneous Poisson
process with intensity~$\eta \be |x|^{-s} \in (0,\infty)$ and, as is readily checked, there is almost surely a unique pair~$(X,Y) \in \scrI \cup \scrI'$ that minimizes the function
\begin{equation}
\label{E:3.14}
    f_x(\tilde x, \tilde y) := \bigl( |x|^{-\gamma_1} |\tilde x| \bigr)^{2d}
    +  \bigl( |x|^{-\gamma_2} |\tilde y-x| \bigr)^{2d} .
\end{equation}
The joint law of~$X$ and~$Y$ can be computed explicitly
\begin{equation}
\label{E:3.14a}
P\bigl((X,Y)\in B\bigr) = \frac{\eta\beta}{|x|^s}\int_B\exp\Bigl\{-\frac{\eta\beta}{|x|^s}\int 1_{\{f_x( x', y')\le f_x(\tilde x,\tilde y)\}}\textd x'\textd y'\Bigr\}\,\textd\tilde x\textd\tilde y.
\end{equation}
The random variables
\begin{equation}
\label{E:3.15}
    Z := |x|^{-\gamma_1} X
    \quad \text{and} \quad
    Z' := |x|^{-\gamma_2} (Y-x)
\end{equation}
then have the joint law
\begin{equation}
%\label{}
P\bigl((Z,Z')\in B\bigr) = \frac{\eta\beta}{|x|^{\zeta}}\int_B\exp\biggl\{-\frac{\eta\beta}{|x|^{\zeta}}\int 1_{\{|\tilde z|^{2d}+|\tilde z'|^{2d}\le|z|^{2d}+|z'|^{2d}\}}\textd\tilde z\textd\tilde z'\biggr\}\textd z\textd z'\,,
\end{equation}
 where  $\zeta:= s-d(\gamma_1+\gamma_2)$. Scaling~$\tilde z$ and~$\tilde z'$ by $(|z|^{2d}+|z'|^{2d})^{1/d}$, the  inner  integral is shown to equal $c_0(|z|^{2d}+|z'|^{2d})$   thus turning the outer integral into one with respect to a product measure.  (This is where using $2d$-powers in \eqref{E:3.14} is crucial.)
 Noting that  $\zeta = 0$ by assumption, $Z$ and~$Z'$ are independent with above law. 

Next define the event~$A(x)$ as follows:
When~$|x|$ is large enough (with  the exact restrictions including the bounds  mentioned above  as well as those noted  below), set
\begin{equation}
\label{E:3.16}
    A(x) := \bigl\{ |Z| > |x|^{\overline\gamma - \gamma_1} \bigr\} \cup
        \bigl\{ |Z'| > |x|^{\overline\gamma - \gamma_2} \bigr\} 
\end{equation}
and let $A(x)$ be the entire probability space otherwise. 
On the event~$A(x)^\cc$ the edge~$(X,Y)$ lies in~$\scrI$
since~$|X| \vee |Y-x| \le |x|^{\overline\gamma}$ and so $(\lfloor X\rfloor,\lfloor Y\rfloor)\in\scrE$ by \eqref{E:3.11}. Moreover, both~$X$ and~$Y$ are within distance~$2|x|$ of the origin
(as long as~$|x|$ is large enough; this is part of the bounds on~$|x|$). 
Recalling the notation $B(y,r):=\{z\in\R^d\colon |z-y|<r\}$, similar arithmetic as in~\cite[eq.~(2.29) and~(2.30)]{Biskup-Lin}  shows 
\begin{equation}
%\label{}
B\bigl(0,2|\lfloor X\rfloor|\bigr) \subseteq B\bigl(0,2|x|\bigr)\,\,\wedge\,\,B\bigl(x,2|\lfloor Y\rfloor-x|\bigr) \subseteq B\bigl(0,2|x|\bigr).
\end{equation}
Picking a path achieving $\wt D(0,\lfloor X\rfloor)$, concatenating it with edge $(\lfloor X\rfloor,\lfloor Y\rfloor)$ and a path achieving $\wt D(x,\lfloor Y\rfloor)$ then produces a path in $B(0,2|x|)$ whose length dominates the restricted distance $\wt D(0,x)$. 

Using  \eqref{E:3.2} to bound $\wt D(0,x)$ by $|x|_11_{A(x)}$ when~$A(x)$ occurs, this yields the pointwise inequality
%(cf.\ \cite[eq.~(2.27)]{Biskup-Lin}):
\begin{equation}
    \wt D(0,x)
    \le \wt D\bigl( 0,\lfloor |x|^{\gamma_1} Z\rfloor \bigr)
    + \wt D \bigl( x, x + \lfloor |x|^{\gamma_2} Z' \rfloor \bigr)
    + 1 + |x|_1 1_{A(x)} .
\end{equation}
In light of \eqref{E:3.13}, the two instances of~$\wt D$ on the right can be regarded as independent of each other and of the variables~$Z$ and~$Z'$. Invoking translation invariance of the law of~$\wt D$, the proof is reduced to \eqref{E:A-prob}. This follows readily from \eqref{E:3.16} and \eqref{E:z-law}.
\end{proofsect}

\subsection{Convergence for restricted distance}
The next several steps hew closely to the original  argument from \cite{Biskup-Lin}. 
Indeed, taking expectation in \eqref{E:subadd} with $\gamma_1=\gamma_2=\gamma$ gives
\begin{equation}
\label{E:3.19}
E\wt D(0,x)\le 2 E\wt D\bigl(0,\lfloor |x|^\gamma Z\rfloor\bigr)+1+|x|_1 P\bigl(A(x)\bigr).
\end{equation}
In order to unite the arguments in the two expectations and get an expression that can be iterated, we replace~$x$ by the random variable
\begin{equation} \label{E:w-def}
    W := Z_0 \prod_{k=1}^\infty |Z_k|^{\ga^k} ,
\end{equation}
where~$Z_0,Z_1,\dotsc$ are i.i.d.\ copies of~$Z$. As shown in \cite[Lemma~3.1]{Biskup-Lin}, the infinite product converges and $W\in(0,\infty)$ a.s., with~$W$ admitting a continuous, a.e.-non-vanishing probability density and finite moments of all orders. Noting that for~$W$ and~$Z$ independent we get $|W|^\ga Z \laweq W$, taking~$W$ independent of the~$\wt D$'s then yields
\begin{equation}
%\label{}
E\wt D(0,rW)\le 2E\wt D(0,r^\gamma W)+c
\end{equation}
for $c:=1+\sup_{x\in\R^d}|x|_1P(A(x))$. This implies the existence of the limit
\begin{equation} \label{E:L-def}
  L_\be(r) := \lim_{n \to \infty}
        \frac{ E \wt D \bigl( 0, \lfloor r^{\ga^{-n}} W \rfloor \bigr) }{
            2^n}
\end{equation}
giving us
\begin{equation}
\label{E:3.23}
\forall r>1\colon\quad \phi_\beta(r):=L_\beta(r)(\log r)^{-\Delta}
\end{equation}
From \eqref{E:L-def} we get $ 2 L_\beta(r^\gamma)=L_\beta(r)$, which then forces the log-log-periodicity \eqref{E:1.3a}. 
The construction via a (essentially) decreasing limit then ensures that~$\phi_\beta$ is bounded from above on~$(1,\infty)$  and  \cite[Theorem~2.5]{Biskup-Lin}  shows  that~$\phi_\beta$ is also uniformly positive.

While simple, the construction of~$L_\beta$ via \eqref{E:L-def} harbors several conceptual problems. First, it concerns the restricted distance. Second, it depends on~$W$ which itself depends on~$\beta$ and~$\eta$. In~\cite[Section~3]{Biskup-Lin}, these concerns are dispelled by subsequently proving that,
for all $r \ge 1$ and Lebesgue a.e.~$x \in \R^d$,
\begin{equation} 
\label{E:dist-1-lim-as}
    \frac{\wt D \bigl( 0, \lfloor r^{\ga^{-n}} x \rfloor \bigr)}{2^n}
    \:\underset{n \to \infty}\longrightarrow \: L_\beta(r) ,
    \quad \text{$P$-a.s.}
\end{equation}
see~\cite[Proposition~3.3]{Biskup-Lin}. The proof of this is based on the subadditivity estimate \eqref{E:subadd} and, modulo rounding of the arguments of~$\wt D$, it can be taken over \emph{verbatim}.

Another concern is the regularity of~$r\mapsto L_\beta(r)$. As in~\cite{Biskup-Lin}, this can again be handled using the subadditivity bound \eqref{E:subadd} which gives
\begin{equation}
\label{E:3.27}
\wt D\bigl(0, \lfloor r^{\ga^{-n}}x\rfloor\bigr)
    \overset{\text{\rm law}}\le 
    \wt D(0, \lfloor r^{\ga_1\ga^{-n}} |x|^{\ga_1} Z\rfloor\bigr)
    + \wt D'\bigl(0, \lfloor r^{\ga_2\ga^{-n}} |x|^{\ga_2} Z'\rfloor\bigr)
    +O(1),%+ 1 + r^{\ga^{-n}} |x| 1_{A(r^{\ga^{-n}}x)} .
\end{equation}
where, thanks to \eqref{E:A-prob}, $O(1)$ is bounded in~$L^1$ uniformly in~$x$ and~$r\ge1$. Since~$Z$ is continuously distributed, \twoeqref{E:dist-1-lim-as}{E:3.27} give
\begin{equation}
%\label{}
L_\beta(r)\le L_\beta(r^{\gamma_1})+L_\beta(r^{\gamma_2})
\end{equation}
for all $\gamma_1,\gamma_2 \in  (0,  \overline\gamma )$ with~$\gamma_1+\gamma_2=2\gamma$. This implies convexity of~$t\mapsto L_\beta(\texte^t)$ and thus continuity of~$r\mapsto L_\beta(r)$ and $r\mapsto\phi_\beta(r)$ on~$(1,\infty)$. 

The next step in the argument is the replacement of  the  limit along doubly exponentially growing sequences by a plain limit $r\to\infty$. This comes at the cost of reinserting~$W$:

\begin{lemma} 
\label{P:dist-1-plim}
Suppose~$\wt D$ and~$W$ are independent. Then
\begin{equation}
    \frac{\wt D \bigl( 0,\lfloor rW \rfloor \bigr)}{L(r)}
    \,\, \underset{r \to \infty}\longrightarrow \,\,
    1, \quad \text{\rm in probability and in~$L^2$} .
\end{equation}
\end{lemma}

\begin{proofsect}{Proof}
The corresponding statement in~\cite{Biskup-Lin} (see Proposition~3.7 there) is deduced from the fact that, for $X_n(r):=2^{-n}\wt D(0,r^{\gamma^{-n}}W)$, the limits $EX_n(r)\to L_\beta(r)$ and $\Var(X_n(r))\to0$ are locally uniform in~$r\ge1$. This is in turn proved by noting that, thanks to \eqref{E:subadd}, both $EX_n(r)$ and~$E(X_n(r)^2)$ are downward monotone in~$n$ modulo additive correction terms  that vanish as $n\to\infty$ locally uniformly in $r>1$.  As~$r\mapsto X_n(r)$ is continuous in the continuum model, the local uniformity is then extracted from Dini's Theorem.

In order to adapt this reasoning to our setting, we need to supply an argument for continuity. This can be achieved by extending the definition of~$x\mapsto\wt D(0,x)$ to all~$x\in\R^d$ as follows: Let $\dist_\infty$ denote the $\ell^\infty$-distance on~$\R^d$ and let $h\colon[0,1]^d\times\{0,1\}^d\to[0,1]$ be defined by
\begin{equation}
%\label{}
h(x,\sigma) :=\bigl[1-\dist_\infty(x,\sigma)\bigr]\Bigl(\sum_{\sigma'\in\{0,1\}^d}\bigl[1-\dist_\infty(x,\sigma')\bigr]\Bigr)^{-1}
\end{equation}
This function is continuous in~$x$ with $h(\sigma,\sigma')=\delta_{\sigma,\sigma'}$ for all $\sigma,\sigma'\in\{0,1\}^d$.  Defining, for each $x\in\R^d$,  
\begin{equation}
%\label{}
\wt D(0,x):=\sum_{\sigma\in\{0,1\}^d}h\bigl(x-\lfloor x\rfloor,\sigma\bigr)\wt D\bigl(0,\lfloor x\rfloor+\sigma\bigr),
\end{equation}
 we get an continuous extension of~$x\mapsto \wt D(0,x)$ to~$\R^d$. 
The subadditive bound \eqref{E:subadd} (which implied the aforementioned downward monotonicity) holds without any rounding albeit with ``$1$'' on the right replaced by a $d$-dependent constant thanks to the bound $|\wt D(0,x)-\wt D(0,\lfloor x\rfloor)|\le d$ for all~$x\in\R^d$.
This constant is irrelevant in the argument and so we can  proceed as in \cite{Biskup-Lin}.
\end{proofsect}

Before we move on, we record a useful consequence of above derivations:

\begin{corollary}
\label{cor-3.3}
There is $c =c(d,s,\beta) \in(0,\infty)$ such that
\begin{equation}
\label{E:3.30}
\forall x\in \Z^d\smallsetminus\{0\}\colon\quad
E \bigl(\wt D(0,x)\bigr)\le c\bigl[1+\log|x|\bigr]^\Delta.
\end{equation}
\end{corollary}

\begin{proofsect}{Proof}
 The claim will follow from \eqref{E:3.19} and the bound \eqref{E:A-prob} once we prove
\begin{equation}
\label{E:3.33}
\forall r\ge1\colon\quad E \wt D(0,rZ)\le \tilde c(1+\log r)^\Delta
\end{equation}
with~$\tilde c\in(0,\infty)$  a constant and~$Z\independent\wt D$. For this we note that, by  Lemma~\ref{P:dist-1-plim}, a bound of this kind holds for~$Z$ replaced by~$W$ so it suffices to ``exchange'' the probability density~$f_W$ of~$W$ for that of~$Z$. Using that $Z\laweq W/|W'|^\gamma$ for $W\independent W'$ with $W'\laweq W$ and writing $f_Z(z):=\sqrt{\eta\beta}\,\texte^{-\eta c_0|z|^{2d}}$ for the probability density of~$Z$, we have
\begin{equation}
%\label{}
f_W(w) = \int_{(0,\infty)} v^{-1}f_Z(w/v)\mu(\textd v),
\end{equation}
where $\mu$ is the law of~$|W|^\gamma$ on~$\R$. For~$v\ge1$, we have $f_Z(w/v)\ge  f_Z(w)$ and so it suffices to show that~$\mu([1,v_0])>0$ for some~$v_0>1$.  By continuity of measure,  for this it suffices to  have  $\mu([1,\infty))>0$ which is checked readily from \eqref{E:w-def} and \eqref{E:z-law}.
\end{proofsect}

\subsection{Actual distance}
We are now ready to start working towards the asymptotics of the actual distance~$D$. Paralleling the approach in~\cite[Section~4]{Biskup-Lin}, fix $\overline\gamma\in(\gamma,1)$ and extend~$\wt D$ to a family of restricted ``distance'' functions,
\begin{equation} \label{E:dist-k}
    \wt D_k(x,y) := \min \left\{ n \ge 0 \colon \,
        \begin{gathered}
            \{(x_{i-1},x_i) \colon i=1,\dotsc,n\} \subseteq \scrE , \:
            x_0 = x, \:
            x_n = y, \\
            \forall i=1,\dots,n\colon\, |x_i-x|
             \le 2|x-y|^{{\overline\gamma}^{\,-k}} \\
        \end{gathered} \right\}.
\end{equation}
These interpolate between the actual distance and the restricted distance monotonically,
\begin{equation} \label{E:dist-k-mono} \begin{aligned}
    D(x,y)
    &\le \dotsb
    \le \wt D_{k+1}(x,y)
    \le \wt D_k(x,y) \\
    &\le \dotsb
    \le \wt D_1(x,y)
    \le \wt D_0(x,y)
    = \wt D(x,y) .
\end{aligned} 
\end{equation}
Since~$k \mapsto \wt D_k(x,y)$ is non-increasing, non-negative,
and takes values in~$\Z$, the sequence $\{\wt D_k(x,y)\}_{k\ge1}$ must stabilize;
i.e.,\ $\wt D_k(x,y) = \wt D(x,y)$ for all~$k$ sufficiently large,
depending on~$x$, $y$, and on the random edges that determine the distances.
A key fact is that, at large scales, this happens uniformly with high probability:

\begin{lemma}
\label{lemma-3.3}
Let~$W$ be independent of the distances~$\wt D_k$ and~$D$.
There is $k \in \N$ such that
\begin{equation}
    \lim_{r \to \infty} P \Bigl(
        \wt D_k \bigl( 0,\lfloor rW \rfloor \bigr)
        = D \bigl( 0, \lfloor rW \rfloor \bigr)
    \Bigr) = 0 .
\end{equation}
\end{lemma}

\begin{proofsect}{Proof}
This is a lattice version of \cite[Lemma~4.2]{Biskup-Lin} whose proof went through by way of the discrete distances and so can be taken over without change.
\end{proofsect}

The next result to establish is an analogue of~\cite[Lemma~4.3]{Biskup-Lin}, which bounds
the ratio~$E\wt D_k(0,\lfloor rW\rfloor)/L(r)$ asymptotically by one from below. In~\cite{Biskup-Lin}, the proof relied on continuity of~$\beta\mapsto\phi_\beta(r)$ which was in turn proved using scaling arguments that do not seem to apply here. However, the above does give us the following:

\begin{lemma} \label{L:cts-ae-be}
For each~$r>1$, $\be\mapsto\phi_\be(r)$ is left-continuous and downward monotone.
There exists an (at most) countable set~$\Sigma \subseteq (0, \infty)$
such that, for each~$r>1$,
the function~$\be \mapsto \phi_\be(r)$ is continuous
at all points~$\be' \in (0,\infty) \smallsetminus \Sigma$.
\end{lemma}

\begin{proof}
In light of \eqref{E:3.23}, the downward monotonicity follows from \eqref{E:dist-1-lim-as} and the fact that, under a monotone coupling of edge sets for two different~$\beta$, distances are ordered pointwise. Being a downward limit of continuous functions, $\beta\mapsto L_\beta(r)$ is left-continuous, and hence so is $\beta\mapsto\phi_\beta(r)$.

The convexity of~$t\mapsto L_\beta(\texte^t)$  shown  above  guarantees that, for any
$0<\be_0<\be_1<\infty$ and $1<r_0<r_1<\infty$,
    the family of functions
\begin{equation}
\label{E:3.34}
    \bigl\{ r \mapsto L_\be(r) \colon \be \in [\be_0,\be_1] \bigr\}
\end{equation}
is uniformly equicontinuous on~$[r_1,r_2]$. This implies that,
if $\be \mapsto \phi_\be(r)$ is continuous
at some~$\be' \in (0,\infty)$ for all~$r \in \Q \cap [r_1,r_2]$,
then it is continuous at~$\be'$ for all~$r \in [r_1,r_2]$.
Invoking the log-log-periodicity \eqref{E:1.3a}, $\beta\mapsto\phi_\beta(r)$ is  non-increasing and  continuous for all~$r>1$ as soon as~$\beta$ does not belong to
\begin{equation}
    \Sigma \: := \bigcup_{r \in \Q \cap [\texte^{\gamma}, \texte]}
        \Bigl\{ \be \in (0,\infty) \colon
            \lim_{\be' \downarrow \be} \phi_{\be'}(r)
            > \lim_{\be' \uparrow \be} \phi_{\be'}(r) \Bigr\}.
\end{equation}
This set is (at most) countable,
since for each~$r \in \Q \cap [\texte^\ga,\texte]$
the set of jump discontinuities of~$\be \mapsto \phi_\be(r)$
is at most countable.
\end{proof}

All that  we needed the  continuity of~$\be \mapsto \phi_\be$ for
in~\cite{Biskup-Lin} is condensed into:

\begin{lemma} 
\label{L:ratio-ae-be}
Let~$\Sigma$ be as in Lemma~\ref{L:cts-ae-be}.
Then for each~$\be \not\in \Sigma$,
\begin{equation} \label{E:3.3}
    \lim_{\be' \downarrow \be} \, \inf_{r>1}
        \frac{\phi_{\be'}(r)}{\phi_\be(r)}
    = 1.
\end{equation}
\end{lemma}

\begin{proofsect}{Proof}
We will prove the contrapositive.
First observe that, by the log-log-periodicity~\eqref{E:1.3a}, we may restrict the infimum to $r \in [\texte^\gamma,\texte]$ without changing the result. Next, since the ratio is non-increasing in~$\be'$, we can take~$\be'$ down to~$\be$
along any decreasing sequence~$\be_n \downarrow \be$.
The continuity and boundedness imply existence of a minimizer for each~$n$;
call it~$r_n$ for~$\be' := \be_n$.
By compactness of $[\texte^\gamma,\texte]$
we may assume $r_n\to r_\infty \in [\texte^\gamma,\texte]$ as~$n \to \infty$.
But then the uniform equicontinuity of \eqref{E:3.34} implies
\begin{equation}
    \lim_{\be' \downarrow \be} \, \inf_{r \in [\texte^\gamma,\texte]}
        \frac{\phi_{\be'}(r)}{\phi_\be(r)}
    = \lim_{n \to \infty}
        \frac{\phi_{\be_n}(r_\infty)}{\phi_\be(r_\infty)} .
\end{equation}
If the latter limit is not equal to one,
then $\be' \mapsto \phi_{\be'}(r_\infty)$ is not continuous at~$\be$,
thus forcing $\be \in \Sigma$.
Hence $\be \not\in \Sigma$ implies \eqref{E:3.3}.
\end{proofsect}

Let us henceforth write~$P_\beta$ for the probability and~$E_\beta$ for the expectation associated with edge probabilities~$\frakp_\beta$. We then have:

\begin{proposition} 
\label{L:dist-k-lwr}
Let~$\beta\not\in\Sigma$ and let~$W$ be as in \eqref{E:w-def} for~$Z$ with law \eqref{E:z-law} for~$\eta:=1$. Then
\begin{equation} \label{E:dist-k-lwr}
 \forall k\ge1\colon\quad   \liminf_{r \to \infty} \, \frac{E_\beta\otimes E_W\,\wt D_k(0,rW)}{\phi_\beta(r)(\log r)^\Delta} \ge 1,
\end{equation}
where the  product of expectations indicates that~$W$ and~$\wt D_k$ are independent.
\end{proposition}

\begin{proof}
As in~\cite{Biskup-Lin}, the statement will be deduced from the fact (to be proved) that, for each~$k\ge1$, $\beta>0$ and $\epsilon\in(0, 1/4) $ there is~$c=c(k,\beta,\epsilon)\in(0,\infty)$ such that
\begin{equation} \label{E:subadd2}
    E_\beta\otimes E_W\,\wt D_{k} \bigl( 0, \lfloor r \ep^{-\frac1{2d-s}} W \rfloor \bigr)
    \le 2 E_{\be(1-2\ep)}\otimes E_W\,\wt D_{k+1} \bigl(
        0, \lfloor r^\ga \ep^{-\frac1{2d-s}} W \rfloor \bigr)
    + c .
\end{equation}
Indeed, dividing both sides~$\phi_{\beta(1-2\epsilon)}(r)(\log r)^\Delta$ and taking $r\to\infty$ shows
\begin{equation}
%\label{}
\liminf_{r \to \infty} \, \frac{E_{\beta'}\otimes E_W\,\wt D_{k+1}(0,rW)}{\phi_{\beta'}(r)(\log r)^\Delta}
\ge\biggl[\inf_{r>1}\frac{\phi_\beta(r)}{\phi_{\beta'}(r)}\biggr]
\liminf_{r \to \infty} \, \frac{E_{\beta}\otimes E_W\,\wt D_k(0,rW)}{\phi_{\beta}(r)(\log r)^\Delta},
\end{equation}
where $\beta':=\beta(1-2\epsilon)$. Since \eqref{E:dist-k-lwr} holds for $k:=0$ and all~$\beta>0$ by Lemma~\ref{P:dist-1-plim}, this bounds \eqref{E:dist-k-lwr} inductively for any $k\ge1$ by
\begin{equation}
%\label{}
\prod_{j=1}^k \inf_{r>1}\frac{\phi_{\beta(1-2\epsilon)^{-j}}(r)}{\phi_{\beta(1-2\epsilon)^{1-j}}(r)}\ge\biggl[\inf_{r>1}\frac{\phi_{\beta(1-2\epsilon)^{-k}}(r)}{\phi_{\beta}(r)}\biggr]^k\,,
\end{equation}
where the inequality follows from downward monotonicity of~$\beta\mapsto\phi_\beta(r)$. Taking~$\epsilon\downarrow0$ and applying Lemma~\ref{L:ratio-ae-be} we then get \eqref{E:dist-k-lwr} for all~$\beta\not\in\Sigma$.

As in~\cite{Biskup-Lin}, the proof of~\eqref{E:subadd2} is based on a variant of the argument from Proposition~\ref{P:subadd}. Let~$\mu$, resp., $\mu'$ be as in \twoeqref{E:3.8}{E:3.9} for $\eta:=\epsilon$ and let~$\scrI$, resp., $\scrI'$ be independent Poisson processes with intensity measures $\mu$, resp., $\mu$'. For any~$(\tilde x,\tilde y)\in\R^d\times\R^d$ satisfying $|\tilde x| \le |\tilde y|$, $|\tilde x| \le |x|^{\overline\gamma}$,
and~$|\tilde y-x| \le |x|^{\overline\gamma}$ we have
\begin{equation}
%\label{}
\mu \Bigl( \bigl( \lfloor\tilde x\rfloor + [0,1)^d \bigr)
        \times \bigl( \lfloor\tilde y\rfloor + [0,1)^d \bigr) \Bigr)
        +\frakp_{\beta(1-2\epsilon)}\bigl(\lfloor \tilde x\rfloor,\lfloor \tilde y\rfloor\bigr)
        \le \frakp_{\beta}\bigl(\lfloor \tilde x\rfloor,\lfloor \tilde y\rfloor\bigr)
\end{equation}
provided $|x|$ is sufficiently large. Letting~$\scrE'$ be a sample of edge configuration with probabilities $\frakp_{\beta(1-2\epsilon)}$ which we assume independent of~$\scrI$ and~$\scrI'$, we can couple the above processes to a sample~$\scrE$ of edge configurations with probabilities~$\frakp_\beta$ so that
\begin{equation}
\label{E:3.43}
\bigl\{(\lfloor \tilde x\rfloor,\lfloor \tilde y\rfloor)\colon (\tilde x,\tilde y)\in\scrI\bigr\}\cup\scrE'\subseteq\scrE.
\end{equation}
We then use $\scrI\cup\scrI'$ to pick a pair $(X,Y)$ minimizing \eqref{E:3.14}, define $(Z,Z')$ from these as in \eqref{E:3.15} and $A(x)$ as in \eqref{E:3.16} unless~$|x|$ is small, in which case we set~$A(x)$ to the whole probability space. 

On $A(x)^\cc$ we are guaranteed $(X,Y)\in\scrI$ and so $(\lfloor X\rfloor,\lfloor Y\rfloor)\in\scrE$.  For the same reasons, the  fact that $\overline\gamma<1$  also gives 
\begin{equation}
%\label{}
B\bigl(0,2|X|^{{\overline\gamma}^{\,-(k+1)}}\bigr)\cup B\bigl(x,2|x-Y|^{{\overline\gamma}^{\,-(k+1)}}\bigr)\subseteq B\bigl(0,2|x|^{{\overline\gamma}^{\,-k}}\bigr)
\end{equation}
whenever~$ A(x)^\cc$ occurs.
Writing $\wt D_k'$ for the distances generated by~$\scrE'$ and~$\wt D_k$ for those generated by~$\scrE$, concatenating a path minimizing $\wt D_{k+1}\bigl(0,\lfloor X\rfloor\bigr)$ with edge~$(\lfloor X\rfloor,\lfloor Y\rfloor)$ and the path minimizing $\wt D_{k+1}\bigl(x,\lfloor Y\rfloor\bigr)$ produces a path contributing to the optimization underlying $\wt D_k\bigl(0,x\bigr)$. Thanks to \eqref{E:3.43} we thus get
\begin{equation}
%\label{}
\wt D_k\bigl(0,x\bigr)\le 
\wt D_{k+1}'\bigl(0,\lfloor X\rfloor\bigr)+\wt D_{k+1}'\bigl(x,\lfloor Y\rfloor\bigr)
+1+|x|_11_{A(x)}.
\end{equation}
Rewriting~$X$ and~$Y$ using $Z$ and~$Z'$, plugging for $W$  defined using~$\eta:=\epsilon$  for~$x$  and taking expectation, this yields \eqref{E:subadd2}. As a calculation shows, the change in normalization effectively replaces~$W$ by $\epsilon^{-\frac1{2d-s}}W$.
\end{proof}

We are now ready to give:

\begin{proofsect}{Proof of Theorem~\ref{thm-1}}
Let~$\beta\not\in\Sigma$. Summarizing the above developments, for~$W$ (defined using~$\eta:=1$) independent of~$D$ we have
\begin{equation} \label{E:dist-lim-prob-w}
    \frac{D \bigl( 0, \lfloor rW \rfloor \bigr)}{L_\beta(r)}
    \,\, \underset{r \to \infty}{\longrightarrow} \,\,
    1 \qquad \text{in probability and~$L^2$} .
\end{equation}
Indeed, the upper bound is supplied by Lemma~\ref{P:dist-1-plim} and $D(0,x)\le\wt D(0,x)$, while the lower bound follows from Lemma~\ref{lemma-3.3} and Proposition~\ref{L:dist-k-lwr}.

Fix~$\delta\in(0,1)$. Using that~$W$ admits a probability  density  $f$, the expectation of the quantity on the left of  \eqref{E:1.4a} is bounded by
\begin{equation}
%\label{}
\frac1{r^d}\# B(0,\delta r^d)+ \frac{c_r(\delta)}{\epsilon} E_\beta\otimes E_W\biggl(\Bigl|\frac{D ( 0, \lfloor rW \rfloor )}{L_\beta(r)}-1\Bigr|\biggr),
\end{equation}
where
\begin{equation}
%\label{}
c_r(\delta):=\max_{x\in B(0,r)\smallsetminus B(0,\,\delta r)}\Bigl(\int_{x+[0,1]^d}f(z/r)\,\textd z\Bigr)^{-1}.
\end{equation}
Since~$f$ is continuous  and  positive on $ \R^d\smallsetminus\{0\}$ we have~$\sup_{r\ge 1}c_r(\delta)<\infty$. The second term thus tends to zero as~$r\to\infty$ by \eqref{E:dist-lim-prob-w}. Noting that $r^{-d}\# B(0,\delta r)\le c\delta^d$, the claim follows by taking~$r\to\infty$ and~$\delta\downarrow0$.
\end{proofsect}

\section{Proof of Theorem~\ref{thm-2}: upper bound}
\noindent
Moving to our second main result, we will now put the heuristic derivation from Section~\ref{sec-2.1} on rigorous footing by proving separately upper and lower bounds on~$\phi_\beta$ that match the desired asymptotic in the limit as~$\beta\to\infty$. Here we will prove the upper bound. Throughout we fix~$d\ge1$ and $s\in(d,2d)$ and denote, as before, $\gamma:=\frac{s}{2d}$.

\subsection{Key proposition and preliminaries}
The argument again relies broadly on the subadditivity bound in Proposition~\ref{P:subadd} and subsequent derivations relying on the random variable~$W$. However, the need to include the asymmetric case ($\gamma_1\ne\gamma_2$) robustly and to track all relevant $\beta$-dependent terms explicitly requires additional care. 

Let~$P_\beta$ denote the law of the edges with connection probabilities $\frakp_\beta$ and let~$E_\beta$ denote the associated expectation. Given~$\beta>0$, fix~$\eta\in(0, 1)$ and let~$Z$ henceforth denote the $\R^d$-valued random variable with law
\begin{equation}
\label{E:4.1}
P(Z\in B)=\sqrt\eta\int_B\texte^{-\eta c_0|z|^{2d}}\textd z,
\end{equation}
where $c_0$ is as in \eqref{E:3.5}. 
Given i.i.d.\ copies $\{Z_n\}_{n\ge0}$ of~$Z$, for any sequence $\{\gamma_n\}_{n\ge1}\subseteq [0,1]$ with $\sup_{n\ge1}\gamma_n<1$ let
\begin{equation}
\label{E:4.2a}
W_{\{\gamma_n\}_{n\ge1}}:=Z_0\prod_{n\ge1}|Z_n|^{\prod_{k=1}^n\gamma_k}.
\end{equation}
Since $Z_n\ne0$ a.s.\ with $n\mapsto\log|Z_n|$ growing slower than polynomially while $n\mapsto\prod_{k=1}^n\gamma_k$ decays exponentially, the infinite product converges to a finite and non-zero number a.s. Let
\begin{equation}
%\label{}
\WW:=\biggl\{W_{\{\gamma_n\}_{n\ge1}}\colon \{\gamma_n\}_{n\ge1}\subseteq \Bigl[0%\frac{2\gamma^2}{1+\gamma}
,\frac{2\gamma}{1+\gamma}\Bigr]\biggr\}
\end{equation}
and denote by~$E_W$ the expectation with respect to the law of~$W\in\WW$. Introduce another exponent sequence $\{\vartheta_n\}_{n\ge1}$ by $\vartheta_0:=1$ and
\begin{equation}
\label{E:4.2}
\forall n\ge0\colon \quad\vartheta_{n+1}:=(2\gamma)^{-1}\bigl(\vartheta_{\lfloor n/2\rfloor}+\vartheta_{\lceil n/2\rceil}\bigr).
\end{equation}
Finally, let us use a different way of rounding points in~$\R^d$ by writing $[x]$ to denote the unique~$z\in\Z^d$ with $x-z\in[-1/2,1/2)^d$. (We need this to ensure that $|[x]|_1\le2|x|_1$.) The upper bound in \eqref{E:1.5} will be derived primarily from:

\begin{proposition}
\label{prop-4.3}
There is $\kappa_0>0$ and a function $\chi\colon (0,\infty)\times(0,\infty)\to(0,\infty)$ satisfying
\begin{equation}
\label{E:4.5i}
\lim_{\alpha\to\infty}\,\limsup_{\beta\to\infty}\chi(\alpha,\beta)=0
\end{equation}
such that for all $\alpha,\beta\ge\kappa_0$ with~$\alpha^{-1}\log\beta\ge\kappa_0$ and all $n\ge1$, 
\begin{equation}
\label{E:4.4}
\sup_{W\in\WW}\,
E_\beta\otimes E_W\bigl( D(0,[\texte^{-\alpha\vartheta_n}\beta^{\theta_n} W])\bigr)\le n\bigl[1+\chi(\alpha,\beta)\bigr].
\end{equation}
Here the  product of expectations indicates that~$W$ and~$D(\cdot,\cdot)$ are regarded as independent.
\end{proposition}

The proof requires a number of preliminary considerations. A key input is a version of Proposition~\ref{P:subadd} that carries all the $\beta$-dependent factors explicitly.

\begin{lemma}
\label{lemma-4.4a}
For each $\eta\in(0,1)$ and~$\overline\gamma\in(\gamma,1)$ there is $c= c(\eta,\overline\gamma)\in(0,\infty)$ such that the following holds: Writing~$\wt D$ and~$\wt D'$ for independent copies of the restricted distance under~$P_\beta$ and~$Z$ and~$Z'$ for independent copies of the random variable \eqref{E:4.1}, independent of~$D$ and~$D'$, for all $x\in\R^d$, $n\ge0$, $\beta>0$ and $\gamma_1,\gamma_2\in(0,\overline\gamma)$ with $\gamma_1+\gamma_2=2\gamma$ we have
\begin{multline} 
\label{E:4.5}
\quad
\wt D\bigl( 0, [x]\bigr)
    \,\overset{\text{\rm law}}\le\, \wt D \bigl( 0, [\beta^{\theta_{\lfloor n/2\rfloor}-\gamma_1\theta_{n+1}}|x|^{\gamma_1}Z ] \bigr)
    \\
    + \wt D' \bigl( 0, [\beta^{\theta_{\lceil n/2\rceil}-\gamma_2\theta_{n+1}} |x|^{\gamma_2}Z' ] \bigr)
    + 1 + 2|x|_1 1_{A_{n,\beta}(x)},
    \quad
\end{multline}
where
\begin{equation}
\label{E:4.7}
A_{n,\beta}(x):=\Bigl\{|Z|>\beta^{-\theta_{\lfloor n/2\rfloor}+\gamma_1\theta_{n+1}}|x|^{\overline\gamma-\gamma_1}\Bigr\}\cup\Bigl\{|Z'|>\beta^{-\theta_{\lceil n/2\rceil}+\gamma_2\theta_{n+1}}|x|^{\overline\gamma-\gamma_2}\Bigr\}
\end{equation}
when $|x|> c(\eta,\overline\gamma)(1+\beta^{1/s})$ and $A_{n,\beta}(x)$ is the whole probability space otherwise.
\end{lemma}

\begin{proofsect}{Proof}
Let~$\beta>0$, $\eta\in(0,1)$ and~$\overline\gamma\in(\gamma,1)$. Modulo some trivial modifications due to a different way of rounding, the early parts of the proof run almost exactly as in Proposition~\ref{P:subadd}. The measures $\mu$ and~$\mu'$ are just as in \twoeqref{E:3.8}{E:3.9} and the processes~$\scrI$ and~$\scrI'$ are as before. The requirement that~$|x|$ is  larger  than an $\eta,\overline\gamma$-dependent constant still ensures \eqref{E:3.12}. To get \eqref{E:3.11} we need
\begin{equation}
%\label{}
\mu \Bigl( \bigl( \lfloor\tilde x\rfloor + [0,1)^d \bigr)
        \times \bigl( \lfloor\tilde y\rfloor + [0,1)^d \bigr) \Bigr)
        \le\frakp_\beta\bigl(\lfloor\tilde x\rfloor-\lfloor\tilde y\rfloor\bigr)
\end{equation}
whenever $(\tilde x,\tilde y)\in\scrI$ and, in particular, whenever $(\tilde x,\tilde y)$ lies in the support of~$\mu$. It is here we need to assume that $\eta<1$ and that $|x|$ exceeds an $\eta,\overline\gamma$-dependent  multiple of~$\beta^{1/s}$. We will write  the combined restrictions on~$x$  as $|x|\ge c(\eta,\overline\gamma)(1+\beta^{1/s})$.

 The most  important change comes in the construction of the pair~$X$ and~$Y$. Instead of \eqref{E:3.14}, here we optimize the function
\begin{equation}
%\label{}
f_{\beta,n,x}(\tilde x, \tilde y) := \bigl(\beta^{-\theta_{\lfloor n/2\rfloor}+\gamma_1\theta_{n+1}} |x|^{-\gamma_1} |\tilde x| \bigr)^{2d}
    +  \bigl(\beta^{-\theta_{\lceil n/2\rceil}+\gamma_2\theta_{n+1}} |x|^{-\gamma_2} |\tilde y-x| \bigr)^{2d} .
\end{equation}
Using this instead of~$f_x$ in the formula for the joint law \eqref{E:3.14a} and setting
\begin{equation}
%\label{E:3.15}
    Z := \beta^{-\theta_{\lfloor n/2\rfloor}+\gamma_1\theta_{n+1}}|x|^{-\gamma_1} X
    \quad \wedge \quad
    Z' := \beta^{-\theta_{\lceil n/2\rceil}+\gamma_2\theta_{n+1}}|x|^{-\gamma_2} (Y-x),
\end{equation}
with the help of $\gamma_1+\gamma_2=2\gamma$ and $\theta_{n+1}=\frac1s+\frac ds(\theta_{\lfloor n/2\rfloor}+\theta_{\lceil n/2\rceil})$ we then verify that~$Z$ and~$Z'$ are independent with law \eqref{E:4.1}. 
Definining~$A_{n,\beta}$ as in the statement and invoking $|[x]|_1\le 2|x|_1$, the rest of the proof of Proposition~\ref{P:subadd} applies to give us \eqref{E:4.5}.
\end{proofsect}

In order to use \eqref{E:4.5} fruitfully, we will need the following observations about the exponent sequence~$\{\theta_n\}_{n\ge1}$ and the auxiliary sequence~$\{\vartheta_n\}_{n\ge1}$:

\begin{lemma}
%\label{lemma}
$\{\vartheta_n\}_{n\ge1}$ is strictly increasing with
\begin{equation}
\label{E:4.3}
%\frac{2\gamma^2}{1+\gamma}=\frac{\vartheta_0}{\vartheta_2}=\inf_{n\ge0}\frac{\vartheta_{\lfloor n/2\rfloor}}{\vartheta_{n+1}}
%\le
\sup_{n\ge0}\frac{\vartheta_{\lceil n/2\rceil}}{\vartheta_{n+1}}= \frac{\vartheta_1}{\vartheta_2}=\frac{2\gamma}{1+\gamma} .%= \frac{2s}{2d+s}.
\end{equation}
In addition, we also have
\begin{equation}
\label{E:theta-ratio}
\sup_{n\ge0}\frac{\theta_{\lceil n/2\rceil}}{\theta_{n+1}}=\gamma
\end{equation}
and
\begin{equation}
\label{E:4.5a}
\sup_{n\ge1}\frac{\vartheta_n}{\theta_{n}}=\max\Bigl\{\frac{\vartheta_1}{\theta_1},\frac{\vartheta_2}{\theta_2}\Bigr\}<\infty.
\end{equation}
\end{lemma}

\begin{proofsect}{Proof}
While the values of~$\{\vartheta_n\}_{n\ge1}$ can be computed explicitly following similar arguments as in the proof of Lemmas~\ref{lemma-3.2}--\ref{lemma-2}, for above claims we will only need the following: Using \eqref{E:4.2} we get
\begin{equation}
%\label{}
\forall n\ge0\colon\quad \vartheta_{2n+1}-\vartheta_{2n} = (2\gamma)^{-1}(\vartheta_{n+1}-\vartheta_n)
\end{equation}
and
\begin{equation}
%\label{}
\forall n\ge1\colon\quad \vartheta_{2n}-\vartheta_{2n-1} = (2\gamma)^{-1}(\vartheta_{n}-\vartheta_{n-1})
\end{equation}
which, in light of~$2\gamma>1$ and $\vartheta_1=\gamma^{-1}>1=\vartheta_0$, shows that~$n\mapsto\vartheta_{n+1}-\vartheta_n$ is positive and non-increasing and thus $n\mapsto\vartheta_n$ is positive and strictly increasing. 

For \eqref{E:4.3} we then use \eqref{E:4.2} to write 
\begin{equation}
\label{E:4.7a}
\frac{\vartheta_{\lceil n/2\rceil}}{\vartheta_{n+1}}= \gamma+\frac{\vartheta_{\lceil n/2\rceil}-\vartheta_{\lfloor n/2\rfloor}}{2\vartheta_{n+1}}.
\end{equation}
The numerator on the right vanishes for~$n$ even while for~$n$ odd it is of the form~$\vartheta_{k+1}-\vartheta_k$, and is thus positive and non-increasing in~$n$. Since the denominator is increasing in~$n$, 
the ratio in \eqref{E:4.7a} is maximized by the smallest odd~$n$, which is~$n=1$. A computation then gives \eqref{E:4.3}.

The argument for \eqref{E:theta-ratio} is very similar. Indeed, in light of \eqref{E:2.5}, we have 
\begin{equation}
%\label{}
\frac{\theta_{\lceil n/2\rceil}}{\theta_{n+1}}= \gamma+\frac{\theta_{\lceil n/2\rceil}-\theta_{\lfloor n/2\rfloor}-\frac1d}{2\theta_{n+1}}.
\end{equation}
 Since $\theta_{\lceil n/2\rceil}-\theta_{\lfloor n/2\rfloor}$ is non-negative and maximal at~$n:=1$ where it equals~$1/s$, the second term is non-positive for all~$n\ge1$. Using that~$\theta_{n+1}\to\infty$ as~$n\to\infty$, this gives \eqref{E:theta-ratio}. 

For \eqref{E:4.5a}, we invoke \eqref{E:2.5}, \twoeqref{E:2.3}{E:2.4} and \eqref{E:4.2} to get, for each~$n\ge1$,
\begin{equation}
%\label{}
\frac{\vartheta_{n+1}}{\theta_{n+1}}<
\frac{\vartheta_{\lceil n/2\rceil}+\vartheta_{\lfloor n/2\rfloor}}{\theta_{\lceil n/2\rceil}+\theta_{\lfloor n/2\rfloor}}.
\end{equation}
Thus, if $\vartheta_k\le a\theta_k$ for all~$1\le k\le\lceil n/2\rceil$, then $\vartheta_{n+1}< a\theta_{n+1}$. It follows that the maximum value of the ratio $\vartheta_n/\theta_n$ (for~$n\ge1$) occurs for either~$n=1$ or~$n=2$.
\end{proofsect}

To see why the bound \eqref{E:theta-ratio} is useful, we note:

\begin{corollary}
\label{cor-4.4}
Let~$c(\eta,\overline\gamma)$ and $A_{n,\beta}(x)$ be the constant and the event from Lemma~\ref{lemma-4.4a}. Then for all $x\in\R^d$ with $|x|>c(\eta,\overline\gamma)(1+\beta^{1/s})$,
\begin{equation} 
\label{E:4.6}
P \bigl( A_{n,\beta}(x) \bigr) \le 2\max_{i=1,2}\,P\Bigl(|Z|>\beta^{-(\gamma-\gamma_i)\theta_{n+1}}|x|^{\overline\gamma-\gamma_i}\Bigr)
\end{equation}
hold for all $\beta\ge1$ and $n\ge0$.
\end{corollary}

\begin{proofsect}{Proof}
Using the union bound, this follows from \eqref{E:theta-ratio} and~$\beta\ge1$.
\end{proofsect}

The bounds \eqref{E:4.3} and \eqref{E:4.5a} will in turn be used in the proof of Proposition~\ref{prop-4.3}; specifically, after \eqref{E:4.25} and \eqref{E:4.34}. In addition to bounds on the exponent sequences, we will also need uniform control of the small-valued tail of  random variables~$W\in\WW$:

\begin{lemma}
\label{lemma-4.5a}
There are~$c,\zeta\in(0,\infty)$ such that for all~$r\in[0,1/2)$,
\begin{equation}
\label{E:4.19}
\sup_{W\in\WW}P_W(|W|\le r)\le c\, r^{\zeta}.
\end{equation}
Moreover,  for each natural~$m\ge1$,
\begin{equation}
\label{E:4.20a}
\sup_{W\in\WW} E_W\bigl( |W|^m\bigr)<\infty.
\end{equation}

\end{lemma}

\begin{proofsect}{Proof}
Let~$W\in\WW$ correspond to the sequence~$\{\gamma_n\}_{n\ge1}$. Abbreviate $\wt\gamma:=\frac{2\gamma}{1+\gamma}$ and pick $a\in(\wt\gamma,1)$. Then $\sum_{n\ge0}a^n(1-a)=1$ along with the union bound give, for each~$t\le0$,
\begin{equation}
\label{E:4.21}
\begin{aligned}
P_W\bigl(|W|\le \texte^t\bigr)&=P\biggl(\log|Z_0|+\sum_{n\ge1}\Bigl(\prod_{k=1}^n\gamma_k\Bigr)\log |Z_n|\le t\biggr)
\\
&\le P\bigl(\log|Z_0|\le (1-a)t\bigr)
+\sum_{n\ge1}P\biggl(\Bigl(\prod_{k=1}^n\gamma_k\Bigr)\log|Z_n|\le a^n(1-a)t\biggr)
\\
&\le \sum_{n\ge0}P\Bigl(\log|Z|\le (a/\wt\gamma)^n(1-a)t\Bigr),
\end{aligned}
\end{equation}
where, besides~$t\le0$, we used that~$\gamma_n\le\wt\gamma$ for all~$n\ge1$ in the last step.
Noting that the probability density of~$Z$ is at most~$\sqrt\eta\le1$, there is $\tilde c\in(0,\infty)$ depending only on~$d$ and the norm~$|\cdot|$ such that
\begin{equation}
%\label{}
\forall r\le1\colon\quad
P\bigl(\log|Z|\le (a/\wt\gamma)^n(1-a)\log r\bigr)\le \tilde c\,r^{\,d(a/\wt\gamma)^n(1-a)}.
\end{equation}
It follows that, for~$r\le1/2$, the last sum in \eqref{E:4.21} is dominated by its~$n=0$ term, thus giving us \eqref{E:4.19}  with $\zeta:=d(1-a)$. 

For \eqref{E:4.20a} we dominate  $|W|^m$  by a.s.\ limit of $\prod_{k=0}^n(|Z_k|\vee1)^{m\rho_k}$,  where $\rho_0:=1$ and $\rho_n:=\prod_{k=1}^n\gamma_k$ for~$n\ge1$, whose expectation is bounded using the independence of $\{Z_n\}_{n\ge0}$ and the H\"older inequality (enabled by $\rho_n\in[0,1)$) by  $E(|Z|^m\vee1)$  to power $\sum_{k=0}^n\rho_k$. As  $E(|Z|^m\vee1)<\infty$  and $\sum_{k\ge0}\rho_k<\infty$, the  above finite products converge (as $n\to\infty$)    in the mean by the Monotone Convergence Theorem thus showing \eqref{E:4.20a}.
\end{proofsect}

\subsection{Proof of Proposition~\ref{prop-4.3}}
The proof of Proposition~\ref{prop-4.3} proceeds by induction with the induction step supplied by Lemma~\ref{lemma-4.4a}. There are two base cases, $n=0$ and~$n=1$, of which the latter requires a separate argument and so we address that first:

\begin{lemma}
\label{lemma-4.6}
We have
\begin{equation}
\label{E:4.25a}
\limsup_{\epsilon\downarrow0}\,\limsup_{\beta\to\infty}\,\sup_{W\in\WW}E_\beta\otimes E_W\bigl( \wt D(0,[\epsilon\beta^{1/s} W])\bigr)\le 1.
\end{equation}
\end{lemma}

\begin{proofsect}{Proof}
Given  $p\in(0,1)$,  we split the expectation according to whether~$|W|\le(\beta^{1/s})^{-p}$ or otherwise.
In the former case, we use the inequalities $\wt D(0,[x])\le 2|x|_1$, $|x|_1\le c'|x|$ and \eqref{E:4.19} to get
\begin{multline}
\label{E:4.26}
\quad 
E_\beta\otimes E_W\Bigl( \wt D(0,[\epsilon\beta^{1/s} W])1_{\{|W|\le(\beta^{1/s})^{-p}\}}\Bigr)\le 2\epsilon \beta^{1/s} E\bigl(|W|_11_{\{|W|\le(\beta^{1/s})^{-p}\}}\bigr)
\\
\le 2c'\epsilon (\beta^{1/s})^{1-p}P\bigl(|W|\le(\beta^{1/s})^{-p}\bigr)
\le 2c'c\,\epsilon(\beta^{1/s})^{1-(1+\zeta)p}.
\quad
\end{multline}
Setting $p := \tfrac1{1+\zeta} \in (0,1)$, the right hand side is at most~$2 c \ep$
and vanishes as $\ep \downarrow 0$.

When $|W|>(\beta^{1/s})^{-p}$, using $p < 1$ we get $|\ep \be^{1/s} W| \to \infty$
as~$\be \to \infty$, and in particular $[\ep \be^{1/s} W] \ne 0$.
Hence $\wt D(0, [\ep \be^{1/s} W]) \ge 1$,
and so it suffices to show that
\begin{equation}
\label{E:4.26a}
E_\be \otimes E_W \Bigl( \wt D( 0, [\ep \be^{1/s} W])
1_{\{|W| > (\be^{1/s})^{-p}\}} 1_{\{\wt D(0,[\ep \be^{1/s} W]) > 1\}} \Bigr)
\end{equation}
vanishes in the stated limits.
For this we temporarily replace $W$  by a non-zero deterministic $x\in\R^d$, abbreviate $x_\beta:=[\epsilon\beta^{1/s}x]$ and set 
\begin{equation}
%\label{}
\Lambda:=\Bigl\{z\in\Z^d\colon  \frac14|x_\beta|\le |z|,|z-x_\beta| \le4|x_\beta|\Bigr\}.
\end{equation}
Now observe that, if~$\wt D(0,x_\beta)>1$, then either $\wt D(0,x_\beta)=2$ but the edge $(0,x_\beta)$ is not occupied or no vertex in~$\Lambda$ is connected by an occupied edge to both~$0$ and~$x_\beta$. Hence,
\begin{multline}
\label{E:4.28}
E_\beta\Bigl( \wt D(0,x_\beta)1_{\{\wt D(0,x_\beta)>1\}}\Bigr)
\\
\le 2 P_\beta\bigl((0,x_\beta)\not\in\scrE\bigr)
+2\epsilon\beta^{1/s}|x|_1 P_\beta\Bigl(\forall z\in\Lambda\colon\, (0,z)\not\in\scrE\vee(x_\beta,z)\not\in\scrE\Bigr)
\end{multline}
by the union bound and the deterministic inequality~$\wt D(0,x_\beta)\le |x_\beta|_1\le 2\epsilon\beta^{1/s}|x|_1$.

The first probability on the right of \eqref{E:4.28} is bounded directly as
\begin{equation}
\label{E:4.29}
P_\beta\bigl((0,x_\beta)\not\in\scrE\bigr)=1-\frakp_\beta(0,x_\beta)
=\exp\{-\beta\frakq(x_\beta)\}\le\texte^{-c' \epsilon^{-s} |x|^{-s}}
\end{equation}
for some~$c'\in(0,\infty)$ provided~$|x_\beta|$ exceeds a constant that depends only~$\frakq$.   
In light of the independence of the underlying edge process, the second probability in \eqref{E:4.28} is bounded by
\begin{equation}
%\label{}
\prod_{z\in\Lambda}\bigl(1-\frakp_\beta(0,z)\frakp_\beta(z,x_\beta)\bigr)
\le\exp\Bigl\{-\sum_{z\in\Lambda}\frakp_\beta(0,z)\frakp_\beta(z,x_\beta)\Bigr\}.
\end{equation}
A similar estimate as in \eqref{E:4.29} then gives
\begin{multline}
\label{E:4.32}
\quad
2\epsilon\beta^{1/s}|x|_1
P_\beta\Bigl(\forall z\in\Lambda\colon\, (0,z)\not\in\scrE\vee(x_\beta,z)\not\in\scrE\Bigr)
\\
\le 2\epsilon\beta^{1/s}|x|_1 
\texte^{-c''(\epsilon\beta^{1/s})^d|x|^d (1-\texte^{-c'\epsilon^{-s}|x|^{-s}})^2},
\quad
\end{multline}
where~$c''$ is the constant such that $|\Lambda|\ge c''(\epsilon\beta^{1/s})^d|x|^d$.  

Now substitute~$x:= W$ and note that, since $p<1$, the quantity $x_\beta$ diverges as~$\beta\to\infty$ uniformly on $\{|W|\ge(\beta^{1/s})^{-p}\}$ and the bounds \eqref{E:4.29} and \eqref{E:4.32} are thus in force. Indeed \eqref{E:4.29} equals $\texte^{-c'\epsilon^{-s}|W|^{-s}}$  whose expectation vanishes as~$\epsilon\downarrow0$ in light of the uniform bound on the tails of~$|W|$ supplied by \eqref{E:4.20a} and Chebyshev's inequality.

Concerning the second term on the right of \eqref{E:4.28}, on the event $\{|W|\le\epsilon^{-1}\}$, the right hand side of \eqref{E:4.32} (with $x:=W$) is bounded by a constant times
\begin{equation}
\label{E:4.32a}
\epsilon\beta^{1/s}|W|
\texte^{-\tilde c(\epsilon\beta^{1/s})^d|W|^d},
\end{equation}
where $\tilde c:=c''(1-\texte^{-c'})^2$ and where we used that all norms on~$\R^d$ are comparable. Writing this as $f(\epsilon\beta^{1/s}|W|)$ for~$f(z):=z\texte^{-\tilde c z^d}$, the fact that~$f$ is bounded and vanishes at zero and infinity shows that \eqref{E:4.32a} is bounded and vanishes as~$\beta\to\infty$. The expectation of this term tends to zero as~$\beta\to\infty$ by the Bounded Convergence Theorem.

On the event $\{|W|>\epsilon^{-1}\}$, the right hand side of \eqref{E:4.32} (with $x:=W$) is bounded by a constant times
\begin{equation}
\label{E:4.33}
\epsilon\beta^{1/s}|W|\texte^{-\tilde c'\epsilon^{d-2s}\beta^{d/s}|W|^{d-2s}},
\end{equation}
where $\tilde c':=c''\inf_{z>1}|z|^{2s}(1-\texte^{-c'|z|^{-s}})^2$. Let~$k$ be a natural such that $kd>2s-d$ and note that $r:=d-(2s-d)/k>0$. If $|W|\le(\beta^{1/s})^{1/k}$, then the quantity in \eqref{E:4.33} is at most
\begin{equation}
%\label{}
\epsilon(\beta^{1/s})^{1+1/k}\texte^{-\tilde c'\epsilon^{d-2s}(\beta^{1/s})^r}
\end{equation}
which tends to zero as~$\beta\to\infty$. If in turn $|W|> (\beta^{1/s})^{1/k}$, then the expectation of \eqref{E:4.33} is at most $\epsilon\beta^{1/s}E_W(|W|1_{\{|W|>(\beta^{1/s})^{1/k}})$ which, invoking Chebyshev's inequality, is then at most $\epsilon E_W(|W|^{k+1})$. Thanks to \eqref{E:4.20a}, this tends to zero as~$\epsilon\downarrow0$.

\end{proofsect}

With the preliminaries out of the way, we are ready to give:

\begin{proofsect}{Proof of Proposition~\ref{prop-4.3}}
Since~$\wt D$ dominates~$D$, it suffices to prove the claim for~$\wt D$. For each~$n\ge0$ we first construct a function $\chi_n\colon[0,\infty)\times[1,\infty)\to[0,\infty)$ such that
\begin{equation}
\label{E:4.13}
\sup_{W\in\WW}\,
E_\beta\otimes E_W\bigl( \wt D(0,[\texte^{-\alpha\vartheta_n}\beta^{\theta_n} W])\bigr)\le n+\chi_n(\alpha,\beta)
\end{equation}
and such that~$\chi_{n+1}$ is tied recursively to~$\chi_{\lfloor n/2\rfloor}$ and~$\chi_{\lceil n/2\rceil}$ allowing for inductive control of their~$\beta,\alpha\to\infty$ limits.  These functions depend on parameter $\overline\gamma\in(\frac{2\gamma}{1+\gamma},1)$ that will be fixed throughout this proof. 

For $n=0$ we note that $[x]=0$ unless $|x|_1\ge\frac12$. Using $\wt D(0,x)\le|x|_1$, $|[x]|_1\le 2|x|_1$ and the Chebyshev inequality we then get
\begin{equation}
\label{E:4.24}
E_\beta\otimes E_W\bigl( \wt D(0,[\texte^{-\alpha}W])\bigr)\le 2E_W\Bigl(|W|_11_{\{|W|_1\ge \frac12\texte^{ \alpha}\}}\Bigr)
\le 4 \texte^{-\alpha} E_W\bigl(|W|_1^2\bigr).
\end{equation}
Denoting by $\chi_0(\alpha,\beta)$  the  supremum of the right-hand side over all~$W\in\WW$, which is finite by Lemma~\ref{lemma-4.5a} and the fact that all norms on~$\R^d$ are comparable, we have \eqref{E:4.13} for~$n=0$.
For $n=1$ we simply~put 
\begin{equation}
%\label{}
\chi_1(\alpha,\beta):= \max\Bigl\{0,\sup_{W\in\WW}E_\beta\otimes E_W\bigl( \wt D(0,[\texte^{-\alpha\vartheta_1}\beta^{\theta_1} W])\bigr)-1\Bigr\}.
\end{equation}
which then obeys \eqref{E:4.13} trivially.

Next assuming that, for some~$n\ge1$, the functions~$\chi_0,\dots,\chi_n$ have been constructed, we show how to construct~$\chi_{n+1}$. Pick~$W\in\WW$ and set
\begin{equation}
\label{E:4.25}
\gamma_1:=\frac{\vartheta_{\lfloor n/2\rfloor}}{\vartheta_{n+1}}\,\,\text{ and }\,\,
\gamma_2:=\frac{\vartheta_{\lceil n/2\rceil}}{\vartheta_{n+1}}.
\end{equation}
Then~$\gamma_1,\gamma_2\in(0,\overline\gamma)$ by \eqref{E:4.3} and so Lemma~\ref{lemma-4.4a} along with \eqref{E:4.2} give
\begin{multline}
\label{E:4.15}
\qquad
\wt D\bigl( 0, [\texte^{-\alpha\vartheta_{n+1}}\beta^{\theta_{n+1}}W]\bigr)
    \,\overset{\text{\rm law}}\le\, \wt D \bigl( 0, [\texte^{-\alpha\vartheta_{\lfloor n/2\rfloor}}\beta^{\theta_{\lfloor n/2\rfloor}}|W|^{\gamma_1}Z ] \bigr)
    \\
    + \wt D' \bigl( 0, [\texte^{-\alpha\vartheta_{\lceil n/2\rceil}}\beta^{\theta_{\lceil n/2\rceil}} |W|^{\gamma_2}Z' ] \bigr)
    + 1 + 2|X_{\alpha,\beta}|_1 1_{A_{n,\beta}(X_{\alpha,\beta})},
\qquad
\end{multline}
where $X_{\alpha,\beta}:=\texte^{-\alpha\vartheta_{n+1}}\beta^{\theta_{n+1}}W$ and where $\wt D$, $\wt D'$, $Z$, $Z'$ and~$W$  on the right-hand side are independent with their respective laws. 

We will use \eqref{E:4.15} only when~$|W|$ is sufficiently large so that the bound \eqref{E:4.6} becomes available and useful. Defining ``small'' by~$|W|\le r$ for some~$r\in(0,1/2)$ to be determined, the part of the expectation for~$|W|$ small is then handled by
\begin{equation}
\label{E:4.17}
\begin{aligned}
E_\beta\otimes E_W\Bigl(\wt D\bigl(0,[&\texte^{-\alpha\vartheta_{n+1}}\beta^{\theta_{n+1}}W\rfloor\bigr)
1_{\{|W|\le r\}}\Bigr)
\\
&\le c\bigl[\bigl(\alpha\vartheta_{n+1} +(\log\beta)\theta_{n+1}\bigr)^\Delta+1\bigr]P\bigl(|W|\le r\bigr)
\\
&\le c' \bigl[\bigl(\alpha\vartheta_{n+1} +(\log\beta)\theta_{n+1}\bigr)^\Delta+1\bigr]\,r^{\zeta}.
\end{aligned}
\end{equation}
Here we first took expectation with respect to~$\wt D$ (conditional on~$W$) using Corollary~\ref{cor-3.3}, then used $|[X_{\alpha,\beta}]|\le 2|X_{\alpha,\beta}|$ and~$r\le1/2$ inside the logarithm and, finally, invoked Lemma~\ref{lemma-4.5a}. The bound holds uniformly in~$W\in\WW$.

The part of the expectation for~$|W|$ large is handled via \eqref{E:4.15} with the result
\begin{equation}
\label{E:4.18}
\begin{aligned}
E_\beta\otimes E_W\Bigl(\,&\wt D\bigl( 0, [\texte^{-\alpha\vartheta_{n+1}}\beta^{\theta_{n+1}}W]\bigr)1_{\{|W|> r\}}\Bigr)
\\
&\le E_\beta\otimes E_W\otimes E_Z \Bigl(\wt D \bigl( 0, [\texte^{-\alpha\vartheta_{\lfloor n/2\rfloor}}\beta^{\theta_{\lfloor n/2\rfloor}}|W|^{\gamma_1}Z ] \bigr)\Bigr)\qquad\quad
\\
&\quad\qquad+E_\beta\otimes E_W\otimes E_Z \Bigl(\wt D \bigl( 0, [\texte^{-\alpha\vartheta_{\lceil n/2\rceil}}\beta^{\theta_{\lceil n/2\rceil}} |W|^{\gamma_2}Z ] \bigr)\Bigr)
\\
&\qquad\qquad\qquad+1+2E_W\Bigl(|X_{\alpha,\beta}|_1\, 1_{\{|W|>r\}}P\bigl(A_{n,\beta}(X_{\alpha,\beta})\bigr)\Bigr),
\end{aligned}
\end{equation}
where we also used that~$Z'$ is equidistributed to~$Z$ and~$\wt D'$ is equidistributed to~$\wt D$. In order to control the last expectation, we assume $\texte^{-\alpha\vartheta_{n+1}}\beta^{\theta_{n+1}}r\ge c(\eta,\overline\gamma)(1+\beta^{1/s})$  and apply Corollary~\ref{cor-4.4}  to get, for any $x\in\R^d$ with~$|x|\ge r$,
\begin{equation}
%\label{}
\begin{aligned}
P\bigl(A_{n,\beta}(\texte^{-\alpha\vartheta_{n+1}}\beta^{\theta_{n+1}}x)\bigr)
& \le 2 \max_{i=1,2}P_Z\Bigl(|Z|>\beta^{(\overline\gamma-\gamma)\theta_{n+1}}\texte^{-\alpha(\overline\gamma-\gamma_i)\vartheta_{n+1}}|x|^{\overline\gamma-\gamma_i}\Bigr)
\\
& \le 2 P_Z\Bigl(|Z|>\beta^{(\overline\gamma-\gamma)\theta_{n+1}}\texte^{-\alpha\overline\gamma\vartheta_{n+1}}r^{\overline\gamma}\bigl(|x|/r\bigr)^{\overline\gamma-\wt \gamma}\Bigr),
\end{aligned}
\end{equation}
 where we abbreviated $\wt\gamma:=\frac{2\gamma}{1+\gamma}$ and used~$r\le 1$ and $|x|\ge r$ to bound
\begin{equation}
%\label{}
|x|^{\overline\gamma-\gamma_i}=r^{\overline\gamma-\gamma_i}\bigl(|x|/r\bigr)^{\overline\gamma-\gamma_i}\ge r^{\overline\gamma}\bigl(|x|/r\bigr)^{\overline\gamma-\wt \gamma}
\end{equation}
with the help of $\gamma_1,\gamma_2\le\wt\gamma$ as shown in \eqref{E:4.3}.
Letting~$c_1,c_2\in(0,\infty)$ be numbers  such that $P(|Z|>a)\le c_1\texte^{-c_2 a^{2d}}$ for all~$a\ge0$ and writing~$\tilde c\in(0,\infty)$ for the best constant such that $|\cdot|_1\le\tilde c|\cdot|$, the last term on the right of \eqref{E:4.18} is thus bounded by
\begin{equation}
\label{E:4.20}
4 c_1\tilde c\,\texte^{-\alpha\vartheta_{n+1}}\beta^{\theta_{n+1}} 
\sum_{k\ge1} r(k+1)\texte^{-c_2 [\beta^{(\overline\gamma-\gamma)\theta_{n+1}}\texte^{-\alpha\overline\gamma\vartheta_{n+1}}r^{\overline\gamma}]^{2d}k^{2d(\overline\gamma-\wt\gamma)}},
\end{equation}
where the sum arises by partitioning the support of~$W$ according to $|W|/r\in[k,k+1)$ and estimating $P_W(|W|/r\in[k,k+1))$ by one. 
Noting that,  for any $p,q>0$, 
\begin{equation}
%\label{}
\sum_{k\ge1}(k+1)\texte^{-p k^q}\le \int_1^\infty 2z\texte^{-pz^q}\textd z\le \Bigl(2\int_0^\infty z\texte^{-z^q}\textd z\Bigr)p^{-2/q}
\end{equation}
 and denoting the prefactor in large parentheses by~$\hat c(q)$,  the quantity in \eqref{E:4.20} is at most 
\begin{equation}
\label{E:4.22}
4 c_1\tilde c\, \hat c\bigl(2d(\overline\gamma-\gamma)\bigr)c_2^{\frac1{2d(\overline\gamma-\wt \gamma)}}\,\texte^{-\alpha\vartheta_{n+1}}\beta^{\theta_{n+1}} \,r\,\bigl(\beta^{(\overline\gamma-\gamma)\theta_{n+1}}\texte^{-\alpha\overline\gamma\vartheta_{n+1}}r^{\overline\gamma}\bigl)^{-\frac2{\overline\gamma-\wt\gamma}}.
\end{equation}
Again, this holds uniformly in~$W\in\WW$.

We now finally set
\begin{equation}
%\label{}
r:=\bigl(\alpha\vartheta_{n+1} +(\log\beta)\theta_{n+1}\bigr)^{-2\Delta/\zeta}
\end{equation}
and note that, since $\theta_{n+1}>1/s$ for~$n\ge1$,
\begin{equation}
\label{E:4.34}
r\le \frac12\,\,\wedge\,\,\texte^{-\alpha\vartheta_{n+1}}\beta^{\theta_{n+1}}r\ge c(\eta,\overline\gamma)(1+\beta^{1/s})
\end{equation}
hold for all~$n\ge1$ provided that~$\beta$ is sufficiently large and $\alpha^{-1}\log\beta$ exceeds the quantity in \eqref{E:4.5a} by a positive constant. 
Thanks to $\gamma_1,\gamma_2\le\frac{2\gamma}{1+\gamma}$ as seen via \eqref{E:4.3}, we also have~$|W|^{\gamma_i}Z\in\WW$ for both~$i=1,2$ in  \eqref{E:4.18}. The first two expectations on the right of  \eqref{E:4.18} can thus be bounded using \eqref{E:4.13}. In light of $\lfloor n/2\rfloor+\lceil n/2\rceil+1=n+1$, this yields \eqref{E:4.13} for~$n$ replaced by~$n+1$ and
\begin{equation}
\label{E:4.35}
\chi_{n+1}(\alpha,\beta):=\chi_{\lfloor n/2\rfloor}(\alpha,\beta)+\chi_{\lceil n/2\rceil}(\alpha,\beta)+a_n(\alpha,\beta)+b_n(\alpha,\beta),
\end{equation}
where~$a_n(\alpha,\beta)$ is the term on the right-hand side of \eqref{E:4.17} and $b_n(\alpha,\beta)$ is the quantity in \eqref{E:4.22} for~$r$ as above.
Proceeding recursively, this gives \eqref{E:4.13} for all~$n\ge0$.

With \eqref{E:4.13} in hand, it remains to prove \twoeqref{E:4.5i}{E:4.4}.  First observe that, under the assumption~$\alpha,\beta\ge\kappa_0$ and $\alpha^{-1}\log\beta\ge\kappa_0$ with~$\kappa_0$ sufficiently large, $\sup_{n\ge1} a_n(\alpha,\beta)$ tends to zero as~$\beta\to\infty$. Concerning the corresponding statement for~$\sup_{n\ge1}b_n(\alpha,\beta)$, here we observe that, modulo the constant prefactors, \eqref{E:4.22} is reduced to
\begin{equation}
%\label{}
r^{-\frac{\overline\gamma+\wt\gamma}{\overline\gamma-\wt\gamma}}
\,\texte^{\alpha\frac{\overline\gamma+\wt\gamma}{\overline\gamma-\wt\gamma}\vartheta_{n+1}}\,
\beta^{-\theta_{n+1}}.
\end{equation}
Assuming again the above restrictions between~$\alpha$ and~$\beta$, this tends to zero as~$\beta\to\infty$ uniformly in~$n\ge1$. Under the same conditions we thus get that
\begin{equation}
\label{E:4.37}
\wt\chi(\alpha,\beta):=\sup_{n\ge1}\bigl[a_n(\alpha,\beta)+b_n(\alpha,\beta)\bigr]
%\,\underset{\beta\to\infty}\,\longrightarrow\,0.
\end{equation}
obeys 
\begin{equation}
\label{E:4.5k}
\lim_{\alpha\to\infty}\,\limsup_{\beta\to\infty}\wt\chi(\alpha,\beta)=0.
\end{equation}
As $\lfloor n/2\rfloor+\lceil n/2\rceil=n$, induction shows that, for all $n\ge0$,
\begin{equation}
%\label{}
\chi_n(\alpha,\beta)\le (n+1)\max\bigl\{\chi_0(\alpha,\beta),\chi_1(\alpha,\beta)\bigr\}+n\wt\chi(\alpha,\beta).
\end{equation}
Then \eqref{E:4.4} holds with $\chi(\alpha,\beta) := 2\max_{i=0,1}\chi_i(\alpha,\beta)+\wt\chi(\alpha,\beta)$. 
The convergence \eqref{E:4.5i} follows from \eqref{E:4.24}, \eqref{E:4.25a} and \eqref{E:4.5k}. 
\end{proofsect}

\begin{remark}
Note that, since $\theta_1 = 1/s$, the  second condition in  \eqref{E:4.34} is exactly the reason why the case $n=1$ in \eqref{E:4.4} has to be treated separately using Lemma~\ref{lemma-4.6}.
\end{remark}

\subsection{Proof of upper bound in Theorem~\ref{thm-2}}
With Proposition~\ref{prop-4.3} established, we finally give:

\begin{proofsect}{Proof of $\le$ in Theorem~\ref{thm-2}}
Let~$\{\gamma_n\}_{n\ge1}$ be such that $\gamma_n=\gamma$ for all~$n\ge0$ and let~$\eta'\in(0,1)$. For~$\beta\ge1$, the random variable~$W_{\{\gamma_n\}_{n\ge1}}$ for $\eta:=\eta'$ coincides with the random variable in \eqref{E:w-def} provided~$\eta$ in \eqref{E:z-law} is set to~$\eta'/\beta$. By the argument in the proof of Lemma~\ref{P:dist-1-plim}, we then have
\begin{equation}
\label{E:4.40}
2^{-n}E_\beta\otimes E_W(\wt D(0,r^{\gamma^{-n}}W))\,\underset{n\to\infty}\longrightarrow\, L_\beta(r)
\end{equation}
uniformly on compact subsets of~$r\in(1,\infty)$.

Given~$\lambda\in[0,1]$, let $\{\lambda_n\}_{n\ge1}\subseteq[0,1]$ be such that
\begin{equation}
%\label{}
q(n):=\lambda_n (2^n-1)+(1-\lambda_n)(2^{n+1}-1)
\end{equation}
is an integer for each~$n\ge1$ and~$\lambda_n\to\lambda$ as~$n\to\infty$.
Then \twoeqref{E:2.10}{E:2.11} gives 
\begin{equation}
\begin{aligned}
\theta_{q(n)}
&=\lambda_n\theta_{2^n-1}+(1-\lambda_n)\theta_{2^{n+1}-1}
\\
&=\frac1{2d-s}[\lambda+(1-\lambda)\gamma^{-1}+o(1)]\gamma^{-n},
\end{aligned}
\end{equation}
where~$o(1)\to0$ as~$n\to\infty$. 

Let~$\{n_k\}_{k\ge1}$ be a sequence of naturals tending to infinity such that~$v:=\lim_{k\to\infty}\vartheta_{q(n_k)}/\theta_{q(n_k)}$ exists in~$\R$; such a sequence exists thanks to \eqref{E:4.5a}. In light of the aforementioned uniformity in \eqref{E:4.40}, this implies 
\begin{equation}
%\label{}
\lim_{k\to\infty}2^{-n_k}E_\beta\otimes E_W\bigl(\wt D(0,\texte^{-\alpha\vartheta_{q(n_k)}}\beta^{\theta_{q(n_k)}}W)\bigr)=L_\beta\Bigl((\texte^{- \alpha v}\beta)^{[\lambda+(1-\lambda)\gamma^{-1}]\frac1{2d-s}}\Bigr).
\end{equation}
 Noting that  $q(n)=[\lambda+(1-\lambda)2+o(1)]2^n$,  \eqref{E:4.4} in Proposition~\ref{prop-4.3}  then shows
\begin{equation}
\label{E:4.54}
L_\beta\Bigl((\texte^{- \alpha v}\beta)^{[\lambda+(1-\lambda)\gamma^{-1}]\frac1{2d-s}}\Bigr)
\le\bigl[\lambda+2(1-\lambda)\bigr]\bigl[1+\chi(\alpha,\beta)\bigr]
\end{equation}
 whenever  $\alpha$ and~$\beta$ obey the conditions  mentioned there.  
Writing \eqref{E:4.54} using~$\phi_\beta$ and introducing $m(\beta)$ and~$u(\beta)$ as in \eqref{E:1.8a}, the log-log-periodicity \eqref{E:1.3a} gives
\begin{multline}
\label{E:4.55}
\quad
\phi_\beta\Bigl((\texte^{- \alpha v\gamma^{+ m(\beta)}}\texte^{u(\beta)})^{[\lambda+(1-\lambda)\gamma^{-1}]\frac1{2d-s}}\Bigr)
\\
\le\frac{\lambda+2(1-\lambda)}{(\lambda+(1-\lambda)\gamma^{-1})^\Delta}
\frac{(2d-s)^\Delta}{(\log\beta-\alpha v)^\Delta}
\bigl[1+\chi(\alpha,\beta)\bigr].
\quad
\end{multline}
 Varying~$\alpha$ in \eqref{E:4.55}  over scales proportional to~$\log\beta$ (still within restrictions imposed in Proposition~\ref{prop-4.3}) along with log-log periodicity shows   $\sup_{\beta>1}(\log\beta)^\Delta\phi_\beta(r)<\infty$ for all~$r>1$.  The convexity of~$t\mapsto L_\beta(\texte^t)$  then gives 
\begin{equation}
\label{E:4.56}
\bigl\{(\log\beta)^\Delta\phi_\beta\colon\beta\ge1\bigr\}\text{ is uniformly equicontinuous on } [\texte^\gamma,\texte^{\gamma^{-1}}].
\end{equation}
This implies that  the term  $\texte^{- \alpha v\gamma^{+ m(\beta)}}$  (which tends to~$1$ for any fixed~$\alpha$) in \eqref{E:4.55}  is negligible in the limit~$\beta\to\infty$. Invoking \eqref{E:4.5i}, the argument following \eqref{E:2.17} then gives ``$\le$'' in~\eqref{E:1.5}. By \eqref{E:4.56} again, the limit is uniform in~$t$. 
\end{proofsect}

\section{Proof of Theorem~\ref{thm-2}: Lower bound} 
\noindent
It remains to prove the lower bound in \eqref{E:1.5}.
 Here we rely  on an argument from the proof of a corresponding lower bound in~\cite{B2} (which itself goes back to an argument in Trapman~\cite{Trapman}). The main improvement compared to~\cite{B2} is that our computations track the~$\beta$-dependence explicitly and that so via the exponent family~$\{\theta_n\}_{n\ge1}$.

\begin{proposition}
\label{prop-3.1}
Noting that $s<2d$, let~$p\in(0,\infty)$ be such that 
\begin{equation}
\label{E:3.0}
\frac{2d}s p\ge p+s+1.
\end{equation}
Then there are $c,\tilde c, \beta_0\in(0,\infty)$ such that for all~$\beta\ge\beta_0$, all natural~$n\ge1$ and all~$x\in\Z^d\smallsetminus\{0\}$,
\begin{equation}
\label{E:3.1}
P\bigl(D(0,x)\le n\bigr)\le c\Bigl(\beta^{\theta_n}\frac{\texte^{\tilde c n^{1/\Delta}}}{|x|}\Bigr)^s\frac1{n^p}.
\end{equation}
\end{proposition}

\begin{remark}
\cite[Theorem~3.1]{B2} gives the same estimate albeit without the~$\beta^{\theta_n}$- term. As $\theta_n\approx n^{1/\Delta}$, this term can be absorbed into a change of the constant~$\tilde c$.
\end{remark}

Let us first show how Proposition~\ref{prop-3.1}  fits into the proof of  Theorem~\ref{thm-2}:

\begin{proofsect}{Proof of $\ge$ in Theorem~\ref{thm-2}}
We start with some preliminary considerations. Let $n\ge1$ be a large integer and let~$\lambda\in[0,1]$ be such that $\lambda 2^n$ is an integer. For~$\beta>1$ set
\begin{equation}
%\label{}
N_n(\lambda,\beta):=\inf\Bigl\{k\ge0\colon \beta^{\theta_k}\texte^{\tilde c k^{1/\Delta}}\ge\beta^{\theta_{\lambda 2^n+(1-\lambda)2^{n+1}}}\Bigr\},
\end{equation}
where $\tilde c$ is the constant from Proposition~\ref{prop-3.1} and where~$\beta>1$ ensures that the set is non-empty. The  monotonicity of $k\mapsto\theta_k$  gives $N_n(\lambda,\beta)\le\lambda 2^n+(1-\lambda)2^{n+1}$ and
\begin{equation}
%\label{}
\theta_{N_n(\lambda,\beta)}\ge \theta_{\lambda 2^n+(1-\lambda)2^{n+1}} -\tilde c\frac{N_n(\lambda,\beta)^{1/\Delta}}{\log\beta}.
\end{equation}
Next, the concavity and piece-wise linear nature of~$k\mapsto\theta_k$ implies that, for all~$m$ with $2^n-1\le m< 2^{n+1}$ and all $k\le m$,
\begin{equation}
%\label{}
\theta_k\le\theta_m + (k-m)\frac{\theta_{2^{n+1}-1}-\theta_{2^n-1}}{2^n},
\end{equation}
Using this for $k:=N_n(\lambda,\beta)$ and $m:=\lambda 2^n+(1-\lambda)2^{n+1}$ gives
\begin{equation}
%\label{}
N_n(\lambda,\beta)\ge \lambda 2^n+(1-\lambda)2^{n+1} -\frac{\tilde c}{\log\beta}
\frac{N_n(\lambda,\beta)^{1/\Delta}}{\theta_{2^{n+1}-1}-\theta_{2^n-1}}\,2^n.
\end{equation}
 Finally,  $N_n(\lambda,\beta)\le 2^{n+1}$ along with~$2^{-1/\Delta}=\gamma$ implies $N_n(\lambda,\beta)^{1/\Delta}\le\gamma^{-n-1}$ while the explicit form \eqref{E:2.10} shows $\theta_{2^{n+1}-1}-\theta_{2^n-1}= s^{-1}\gamma^{-n}$. Putting this together, we get
\begin{equation}
\label{E:4.19a}
N_n(\lambda,\beta)\ge \bigl[\lambda 2^n+(1-\lambda)2^{n+1}\bigr]\Bigl(1 -\frac{2d\tilde c }{\log\beta}\Bigr),
\end{equation}
where we also used that $2^n\le\lambda 2^n+(1-\lambda)2^{n+1}$.

We now move to the proof of the claim. Fix~$\lambda\in[0,1]$ and let~$\lambda_n\in[0,\lambda]$ be maximal  such  that~$\lambda_n2^n$ is an integer. Since $\theta_n\to\infty$ and $N_n(\lambda_n,\beta)\to\infty$ as~$n\to\infty$, for any~$x\in\R^d$  such that $|x|\ge1$  the bound \eqref{E:3.1} shows that
\begin{equation}
%\label{}
D\bigl(0,[ x\beta^{\theta_{\lambda_n 2^n+(1-\lambda_n)2^{n+1}}}]\bigr) \ge N_n(\lambda_n,\beta)
\end{equation}
occurs with probability tending to one as $n\to\infty$. Bounding $D(\cdot,\cdot)\le\wt D(\cdot,\cdot)$ and substituting $x:=W$, for~$W$ the random variable from \eqref{E:w-def}, the local uniformity of the convergence of $E\wt D(0,r^{\gamma^{-n}}W)/2^n \to L_\beta(r)$ along with the asymptotic  \eqref{E:2.15} show
\begin{equation}
%\label{}
L_\beta\bigl(\beta^{[\lambda+(1-\lambda)\gamma^{-1}]\frac1{2d-s}}\bigr) \ge \bigl[\lambda+(1-\lambda)2\bigr]\Bigl(1 -\frac{2d\tilde c}{\log\beta}\Bigr).
\end{equation}
Modulo the form of the error term, this bound is complementary to \eqref{E:4.54}. The same calculation based on \twoeqref{E:2.17}{E:2.24} then proves ``$\ge$'' in \eqref{E:1.5}.
\end{proofsect}

It remains to give a proof of Proposition~\ref{prop-3.1}.
As in \cite{B2}, we will proceed by induction which will require control of the expected size of balls in the intrinsic metric. Denote  such a ball as 
\begin{equation}
%\label{}
\text{\rm B}(x,k):=\bigl\{z\in\Z^d\colon D(x,z)\le k\bigr\}.
\end{equation}
We then have:

\begin{lemma}
\label{lemma-4.2}
For all $d\ge1$ and $s>d$ there is $a=a(d,s)\in(0,\infty)$ such that if \eqref{E:3.1} holds for some $\beta>0$, $p>0$, an integer~$n\ge1$ and a constant~$c>0$ then
\begin{equation}
%\label{}
E\bigl(|\text{\rm B}(0,n)|\bigr)\le a\bigl(c^{1/s}\beta^{\theta_n} \texte^{\tilde c n^{1/\Delta}} n^{-p/s}\bigr)^d
\end{equation}
\end{lemma}

\begin{proofsect}{Proof}
Given any real~$r>0$, assuming \eqref{E:3.1} we have
\begin{equation}
%\label{}
E\bigl(|\text{\rm B}(0,n)|\bigr)\le\sum_{|x|\le r}1+A^s\sum_{|x|>r}\frac1{|x|^s},
\end{equation}
where $A:=c^{1/s}\beta^{\theta_n}\texte^{\tilde c n^{1/\Delta}} n^{-p/s}$. The right-hand side is bounded by a $d,s$-dependent constant times $r^d+A^s r^{d-s}$. Optimizing over~$r$ then yields the claim.
\end{proofsect}

Another technical input we will need is a bound on a sum that appears in the proof of the induction step:

\begin{lemma}
\label{lemma-4.3}
For all~$d\ge1$, $s\in(d,2d)$ and~$p>0$ there is a constant~$b=b(d,s,p)\in(0,\infty)$ such that for all $\tilde c\ge1$ and all integer~$n\ge1$,
\begin{equation}
\label{E:4.6a}
\sum_{k=0}^{n}\frac{\texte^{\tilde c d[k^{1/\Delta}+(n-k)^{1/\Delta}]}}{(k\vee1)^{pd/s}((n-k)\vee1)^{pd/s}}\le b\frac{\texte^{\tilde cs(n+1)^{1/\Delta}}}{(n+1)^{-1+2pd/s}}.
\end{equation}
Here~$k\vee 1$ is the maximum of~$k$ and~$1$.
\end{lemma}

\begin{proofsect}{Proof}
Using that~$\Delta>1$ we readily check that $x^{1/\Delta}+(1-x)^{1/\Delta}\le 2^{1-1/\Delta}-\kappa(x-1/2)^2$ for some~$\kappa>0$ and all~$x\in[0,1]$. Noting that $d2^{1-1/\Delta}=s$, for   $k$ such that either~$k\le n/3$ or~$k\ge 2n/3$  the numerator is bounded as
\begin{equation}
%\label{}
\texte^{\tilde c d[k^{1/\Delta}+(n-k)^{1/\Delta}]}
\le \texte^{\tilde c s n^{1/\Delta}}\,\texte^{-\frac1{36}d\tilde c\kappa n^{1/\Delta}}.
\end{equation}
The sum in \eqref{E:4.6a} is thus dominated by $k$ in the range~$n/3\le k\le2n/3$, where the numerator at most $\texte^{\tilde c s n^{1/\Delta}}$ while the denominator is at least a constant times $n^{-2dp/s}$. As there are at most~$n$ terms under the sum in this range, the result follows.
\end{proofsect}

Equipped with these technical lemmas, we now prove the induction step:

\begin{lemma}
\label{lemma-4.4}
Suppose that~$\hat c\in(0,\infty)$ is a constant such that,  for each~$\beta\ge1$, 
\begin{equation}
\label{E:5.15}
\forall x\in\Z^d\smallsetminus\{0\}\colon\quad \frakp_\beta(0,x)\le\frac{\hat c\beta}{|x|^s}.
\end{equation}
Then the following holds for all $p>0$, $c>0$ and~$\beta\ge1$ and all integers~$n\ge1$: Assuming that \eqref{E:3.1} is true for all~$x\in\Z^d\smallsetminus\{0\}$ and all positive integers less or equal than~$n$, then
\begin{multline}
\label{E:5.16}
\quad
P\bigl(D(0,x)\le n+1\bigr)
\\
\le
\hat c b\, ac^{d/s}\max\{a c^{d/s},\beta^{1+d\theta_n-s\theta_{n+1}}\}\,\Bigl(\beta^{\theta_{n+1}}\frac{\texte^{\tilde c(n+1)^{1/\Delta}}}{|x|}\Bigr)^s\,(n+1)^{s+1-2pd/s}
\quad
\end{multline}
holds for all~$x\in\Z^d\smallsetminus\{0\}$, where~$a$ and~$b$ are the constants from Lemmas~\ref{lemma-4.2}-\ref{lemma-4.3}.
\end{lemma}

\begin{proofsect}{Proof}
Fix~$n\ge1$ and assume that \eqref{E:3.1} holds for all naturals up to and including~$n$. By the triangle inequality, on $\{D(0,x)\le n+1\}$ every path connecting~$0$ to~$x$ must contain an edge of length at least~$|x|/(n+1)$. Writing $(z,\tilde z)$ for the first edge with this property along the path (as labeled from~$0$ to~$x$), we have
\begin{multline}
%\label{}
\bigl\{D(0,x)\le n+1\bigr\}
\\
\subseteq
\bigcup_{k=0}^{n}\bigcup_{\begin{subarray}{c}
z,z'\in\Z^d\\|\tilde z-z|\ge\frac{|x|}{n+1}
\end{subarray}}
\Bigl(\bigl\{D(0,z)\le k\bigr\}\circ\bigl\{(z,\tilde z)\text{ occupied}\bigr\}\circ \bigl\{D(x,\tilde z)\le n-k\bigr\}\Bigr)
\end{multline}
where $A_1\circ A_2\circ A_3$ is the set of edge-configurations $\omega$ for which there are disjoint finite sets of edges $S_1,S_2,S_3\subseteq\Z^d$ (called witnesses) such that, for each~$i=1,2,3$, the event~$A_i$ occurs in every configuration that agrees with~$\omega$ on~$S_i$. (On $\{D(0,x)\le n+1\}$, these witnesses arise as the corresponding portions of the minimal shortest path, with minimality taken relative to an \emph{a priori} ordering on finite paths on~$\Z^d$. These portions are necessarily edge-disjoint because the path is of minimal length.) Hereby we~get
\begin{multline}
%\label{}
\qquad
P\bigl(D(0,x)\le n+1\bigr)
\\
\le \sum_{k=0}^{n}\sum_{\begin{subarray}{c}
z,z'\in\Z^d\\|\tilde z-z|\ge\frac{|x|}{n+1}
\end{subarray}}\frakp_\beta(z,\tilde z)
P\bigl(D(0,z)\le k\bigr)P\bigl(D(x,\tilde z)\le n-k\bigr),
\qquad
\end{multline}
 employing the  union bound and the van den Berg-Kesten inequality~\cite{vdBK}.

Using \eqref{E:5.15} and ignoring the restriction on the size of $|z-\tilde z|$ we obtain
\begin{equation}
\label{E:3.4}
P\bigl(D(0,x)\le n+1\bigr)
\le\hat c\beta\frac{(n+1)^s}{|x|^s}\sum_{k=0}^{n}E\bigl(|\text{\rm B}(0,k)|\bigr) E\bigl(|\text{\rm B}(0,n-k)|\bigr).
\end{equation}
With the help of the translation invariance of the underlying process, the induction assumption allows us to bound the expectations on the right via Lemma~\ref{lemma-4.2}. Plugging this in \eqref{E:3.4}  while noting that we need to use $|\text{\rm B}(0,0)|=1$ when $k=0,n$ shows 
\begin{multline}
\label{E:3.4b}
\quad
P\bigl(D(0,x)\le n+1\bigr) \le  2\hat c\frac{(n+1)^s}{|x|^s}a c^{d/s}\beta^{1+d\theta_n}\texte^{\tilde c d n^{1/\Delta}} 
\\
+\hat c\frac{(n+1)^s}{|x|^s}(ac^{d/s})^2\sum_{k=1}^{n-1}\beta^{1+d(\theta_k+\theta_{n-k})}\frac{\texte^{\tilde c d(k^{1/\Delta}+(n-k)^{1/\Delta}}}{k^{pd/s}(n-k)^{pd/s}}.
\quad
\end{multline}
Definition~\ref{def-1} gives $1+d(\theta_k+\theta_{n-k})\le s\theta_{n+1}$ for all~$k$ under the sum and so the $\beta$-dependent prefactor is no larger than $\beta^{s\theta_{n+1}}$. (This is where we need $\beta\ge1$.)  Using the maximum in \eqref{E:5.16} to unify the prefactors and subsequently joining the first term on the right \eqref{E:3.4b} with the sum,   the claim follows from Lemma~\ref{lemma-4.3}.
\end{proofsect}

We also need a simple observation about the exponent sequence:

\begin{lemma}
%\label{lemma}
For each $d\ge1$ and $s\in(d,2d)$,
\begin{equation}
\label{E:5.21}
\inf_{n\ge2}\,\bigl[s\theta_{n+1}-1-d\theta_n\bigr]>0.
\end{equation}
\end{lemma}

\begin{proofsect}{Proof}
Thanks to the monotonicity of $k\mapsto\theta_k$, the quantity under infimum is larger than $(s-d)\theta_n-1$, which (in light of $\theta_k\to\infty$) is uniformly positive once~$n$ is sufficiently large. It thus suffices to show that $s\theta_{n+1}-1-d\theta_n>0$ for all~$n\ge2$. By concavity of $k\mapsto\theta_k$ and the explicit representation \twoeqref{E:2.10}{E:2.11}, this is the case as long as $\{1,\dots,n-1\}$ contains a value of the form $2^k-1$ for some natural $k\ge0$. This is true for all~$n\ge2$.
\end{proofsect}

With all the lemmas in hand, we are ready to give:

\begin{proofsect}{Proof of Proposition~\ref{prop-3.1}}
 First observe that, by the fact that $1-\texte^{-\beta\frakq(x)}\le\beta\frak q(x)$, the bound \eqref{E:5.15} is in force for~$x$ large thanks to the assumption \eqref{E:1.1}. Adjusting~$\hat c$ if necessary and assuming~$\beta\ge1$, we then get it for all~$x\ne0$. 

 We will now prove \eqref{E:3.1} by induction on~$n$. There are two base cases: $n=1$ and $n=2$. The former  is settled by the observation that in order for~$x\ne0$ to obey~$D(0,x)\le1$, the edge $(0,x)$ must be occupied. By \eqref{E:5.15} this has probability at most $\hat c\beta|x|^{-s}$ and so, in light of~$\theta_1=1/s$, \eqref{E:3.1} holds for~$n=1$ provided that~$c,\tilde c>0$ obey 
\begin{equation}
\label{E:4.13a}
\hat c\le c\texte^{\tilde c s}.
\end{equation}
 (This holds regardless of $\beta>0$.) As to~$n=2$, here the union bound dominates the probability of $D(0,x)=2$ by
\begin{equation}
%\label{}
\sum_{z\ne0,x}\frakp_\beta(0,z)\frakp_\beta(z,x)\le(\hat c\beta)^2\sum_{z\ne0,x}\frac1{|z|^{{}s}}\frac1{|x-z|^s}.
\end{equation}
Since~$s>d$, the sum on the right is bounded by $\tilde a|x|^{-s}$ for some constant~$\tilde a$ depending only on~$s$ and~$d$. Using the union bound, the claim will hold for~$n=2$ provided that~$c$, $\tilde c$ and~$\beta$ obey
\begin{equation}
%\label{}
\hat c\beta+(\hat c\beta)^2\tilde a\le c \,2^{-p}\,\texte^{\tilde c\,2^{1/\Delta}s}\beta^{s+d},
\end{equation}
where we also used that $s\theta_2=s+d$. As $s+d\ge2$, this is in turn implied by
\begin{equation}
\label{E:5.25}
\hat c+\hat c^2\tilde a\le c\,2^{-p}\texte^{\tilde c\,2^{1/\Delta}s}
\end{equation}
whenever~$\beta\ge1$.

Assuming~$p$ is such that \eqref{E:3.0} holds, once \eqref{E:3.1} is true for all integers up to and including~$n$, Lemma~\ref{lemma-4.4} will extend it to~$n+1$ as soon as
\begin{equation}
\label{E:4.14a}
\hat c b \, ac^{d/s}\max\{a c^{d/s},\beta^{1+d\theta_n-s\theta_{n+1}}\}\le c,
\end{equation}
where~$a$ and~$b$ are as in Lemmas~\ref{lemma-4.2}--\ref{lemma-4.3}.  Writing~$\kappa$ for the infimum \eqref{E:5.21}, this will hold for all $n\ge2$ if
\begin{equation}
%\label{}
\hat c b a^2 c^{2d/s}\le c,
\end{equation}
provided that $\beta\ge\beta_0:=\max\{1,(a c^{d/s})^{-1/\kappa}\}$. This can itself be achieved by taking~$c$ sufficiently small (note that~$2d/s>1$); the inequalities \eqref{E:4.13a} and \eqref{E:5.25} then hold once~$\tilde c$ is sufficiently large. By induction, the claim holds for all~$n\ge1$.
\end{proofsect}

\section*{Acknowledgments}
\nopagebreak\noindent
This work has been partially supported by NSF award DMS-1954343.


\begin{thebibliography}{AX}


\bibitem{Antal-Pisztora}
P.~Antal and A.~Pisztora (1996). On the chemical distance for supercritical Bernoulli percolation. \textit{Annals of Probab.}~\textbf{24} 1036--1048.

\bibitem{50years-FPP}
A.~Auffinger, M.~Damron and J.~Hanson (2017). \textit{50 Years of First-passage Percolation}. Amer. Math. Soc. vol~68, 161 pp.

\bibitem{Benjamini-Berger} 
I.~Benjamini and N.~Berger (2001). 
\textrm{The diameter of long-range
percolation clusters on finite cycles}. 
\textit{Rand. Struct. \& Alg.}~\textbf{19}, no. 2, 102--111.

\bibitem{BBY}
I.~Benjamini, N.~Berger and A.~Yadin (2008). 
{Long range percolation mixing time}. 
\textit{Comb. Probab. Comp.}~\textbf{17}, 487-494.


\bibitem{BKPS} 
I.~Benjamini, H.~Kesten, Y.~Peres and O.~Schramm (2004). {The
geometry of the uniform spanning forests: transitions in
dimensions~4,~8,~12,~\dots}. \textit{Ann. Math.} (2) \textbf{160}, no.~2, 465--491.

\bibitem{vdBK}
J.~van den Berg and H.~Kesten (1985).
Inequalities with applications to percolation and reliability. 
\textit{J. Appl. Prob.} \textbf{22},  556--569.

\bibitem{Berger-RW}
N.~Berger (2002). 
Transience, recurrence and critical behavior for long-range percolation.
\textit{Commun.Math. Phys.} \textbf{226}, no.~3, 531--558.

\bibitem{Berger-LRP}
N.~Berger (2004). 
{A lower bound for the chemical distance in sparse long-range percolation  models}. arxiv:math.PR/0409021.


\bibitem{B1}
M.~Biskup (2004). 
\textrm{On the scaling of the chemical distance in long range percolation models}.  \textit{Ann. Probab.} {\bf 32}, no. 4, 2938--2977.

\bibitem{B2}
M.~Biskup  (2011). 
\textrm{Graph diameter in long-range percolation}. 
\textit{Rand. Struct. \& Alg.} \textbf{39}, no. 2, 210--227.

\bibitem{BCKW}
M. Biskup, X. Chen, T. Kumagai, J. Wang (2021). 
Quenched Invariance Principle for a class of random conductance models with long-range jumps.
 \textit{Probab. Theory Rel. Fields} \textbf{180}, 847--889. 

\bibitem{Biskup-Lin}
M. Biskup and J. Lin (2019). 
Sharp asymptotic for the chemical distance in long-range percolation. 
\textit{Random Struct. \&\ Alg.} \textbf{55}  560--583. 

\bibitem{CCK}
V.H.~Can, D.A.~Croydon and T.~Kumagai (2021). 
Spectral dimension of simple random walk on a long-range percolation cluster. 
arXiv:2111.00718


\bibitem{CD}
S.~Chatterjee and P.~Dey (2016). 
Multiple phase transitions in long-range first-passage percolation on lattices. 
\textit{Commun. Pure Appl. Math.} \textbf{69}, no.~2, 203--256.

\bibitem{CGS} 
D.~Coppersmith, D.~Gamarnik and M.~Sviridenko (2002). 
{The diameter of a long-range percolation graph}. 
\textit{Rand. Struct. \& Alg.} \textbf{21}, no.~1, 1--13.

\bibitem{CS1}
N.~Crawford and A.~Sly  (2012). 
{Simple random walk on long range percolation clusters I: Heat kernel bounds}. 
\textit{Probab. Theory Rel. Fields} \textbf{154}, 753--786.

\bibitem{CS2}
N.~Crawford and A.~Sly  (2013). 
{Simple random walk on long-range percolation clusters II: Scaling limits}.  
\textit{Ann. Probab.}~\textbf{41}, no.~2, 445--502.



\bibitem{DvdHH}
M.~Deijfen, R.~van der Hofstad and~G.Hooghiemstra  (2013). 
Scale-free percolation.
\textit{Ann. Inst. Henri Poincar\'e Probab. Stat.} \textbf{49}, 817--838.

\bibitem{DHW}
P.~Deprez, R.S.~Hazra and M. W\"uttrich (2015).
Inhomogeneous long-range percolation for real-life network modeling.
\textit{Risks} \textbf{3} 1--23.

\bibitem{Ding-Sly}
J.~Ding and A.~Sly (2013). 
Distances in critical long range percolation. 
arXiv:1303.3995

\bibitem{Garet-Marchand}
O.~Garet and R.~Marchand (2004). 
Asymptotic shape for the chemical distance and first-passage percolation on the infinite Bernoulli cluster. 
\textit{ESAIM Probab. Stat.}~\textbf{8}, 169--199

\bibitem{Hamersley-Welsh}
J.~Hammersley and D.~Welsh (1965). 
First-passage percolation, subadditive processes, stochastic networks, and generalized renewal theory. 
\textit{Proc. Internat. Res. Semin., Statist. Lab., Univ. California, Berkeley}, 61--110, Springer-Verlag, New York.

\bibitem{HH}
N.~Hao and M.~Heydenreich (2021). 
Graph distances in scale-free percolation: the logarithmic case. 
arXiv:2105.05709

\bibitem{HHJ}
M.~Heydenreich, T.~Hulshof, J.~Jorritsma (2017). 
Structures in supercritical scale-free percolation.
\textit{Ann. App. Probab.} \textbf{27}, no.~4, 2569-2604

\bibitem{Hutchcroft}
T.~Hutchcroft (2021).
Power-law bounds for critical long-range percolation below the upper-critical dimension,
 \textit{Probab. Theory Rel. Fields} \textbf{181}, 533--570.  
 

\bibitem{Kingman1}
J.F.C.~Kingman (1968). 
The ergodic theory of sub additive stochastic processes.
\textit{J.~Roy. Statist. Soc. Ser.}~\textbf{30}, 499--510.

\bibitem{Kingman2}
J.F.C.~Kingman (1973). 
Subadditive Ergodic Theory. 
\textit{Annals of Probab.}~\textbf{1}, 883--899.

\bibitem{KM}
T.~Kumagai and J.~Misumi (2008). 
Heat kernel estimates for strongly recurrent random walk on random media.
\textit{J. Theor. Probab} \textbf{21}, no.~4, 910--935.

\bibitem{Misumi}
J.~Misumi (2008). 
Estimates of effective resistances in a long-range percolation on~$\Z^d$. 
\textit{J.~Math. Kyoto Univ.},  \textbf{48}, no. 2, 389--40.

\bibitem{Trapman}
P.~Trapman (2010). 
The growth of the infinite long-range percolation cluster. 
\textit{Ann. Probab.} \textbf{38}, 1583--1608.


\end{thebibliography}
\end{document}